\documentclass[10pt,leqno]{amsart}
\usepackage{amsmath,amssymb,amsfonts,mathtools}
\usepackage{amsthm}
\usepackage{graphicx}
\usepackage{textcomp}
\usepackage[T1]{fontenc} 						%
\usepackage[utf8]{inputenc} 				%
\usepackage[english]{babel}
\usepackage{titlesec}

\usepackage[inline,shortlabels]{enumitem}
\usepackage[dvipsnames]{xcolor}
\usepackage{MnSymbol}
\usepackage{tikz-cd}
\usetikzlibrary{hobby}
\usepackage{dsfont}
\usepackage{soul}
\usepackage{subcaption}
\usepackage{indentfirst,csquotes}
\usepackage{hyperref}
\usepackage{fontawesome5}
\usepackage{pifont}

\DeclareMathOperator{\diag}{diag}

\newcommand{\sml}{\scriptscriptstyle}
\newcommand{\qmarks}[1]{``#1''}

\definecolor{myblue}{rgb}{0, 0, 0.8}
\definecolor{myred}{rgb}{0.7, 0, 0}
\definecolor{mygreen}{rgb}{0, 0.5, 0}
\definecolor{my_cayenne}{HTML}{800002}
\definecolor{grapher_red}{HTML}{FB0207}
\definecolor{grapher_green}{HTML}{118040}

\newcommand{\cred}[1]{\color{myred} #1\color{black}}
\newcommand{\cgreen}[1]{\color{mygreen} #1\color{black}}

\newcommand{\stkrl}[2]{\stackrel{\text{\tiny{(#1)}}}{#2}}

\graphicspath{{figs/}}

\newcommand{\gln}{GL(n, \R)}

\newcommand{\repr}[1]{a^{#1}}
\newcommand{\reprarg}[2]{\repr{#1}_{#2}}

\newcommand{\id}[1]{\text{id}_{\sml #1}}
\newcommand{\norm}[1]{\left \|#1 \right \|}
\newcommand{\T}{^\top}

\newcommand{\argmin}{\mathop{\rm argmin}}
\newcommand{\argmax}{\mathop{\rm argmax}}

\newcommand{\R}{\mathbb{R}}
\newcommand{\N}{\mathbb{N}}

\newcommand{\U}{\mathcal{U}}
\newcommand{\X}{\mathcal{X}}

\newcommand{\cost}{r}
\newcommand{\costfb}[1]{\cost_{#1}}

\newcommand{\vv}{v_{\policy, \cost}}
\newcommand{\vvarg}[1]{v_{#1}}
\newcommand{\Vsym}{V}
\newcommand{\V}{\Vsym_{\policy, \cost}}
\newcommand{\Varg}[1]{\Vsym_{#1}}

\newcommand{\q}{q_{\policy, \cost}}

\newcommand{\Q}{Q_{\policy, \cost}}

\newcommand{\Cost}{R}
\newcommand{\Costfb}[1]{\Cost_{#1}}

\newcommand{\f}{f}
\newcommand{\ffb}[1]{\f_{#1}}
\newcommand{\F}{F}
\newcommand{\Ffb}[1]{\F_{#1}}
\newcommand{\Fapx}{\hat{F}}
\newcommand{\Fapxfb}[1]{\Fapx_{#1}}

\newcommand{\policy}{\pi}

\newcommand{\IdxX}{\mathcal{I}_{\sml \X}}
\newcommand{\IdxU}{\mathcal{I}_{\sml \U}}

\newcommand{\Sect}{\Gamma(\X, \U)}

\newcommand{\state}{x} %
\newcommand{\act}{u}   %

\newcommand{\disc}{\gamma}

\newcommand{\advset}[2]{\mathcal{C}_{#1}(#2)}

\newcommand{\feasset}[1]{\mathcal{A}_{\scalebox{0.6}{$\,\disc\!\f$}}(#1)}

\newcommand{\invset}[1]{\mathcal{F}^{\sml I}(#1)}

\newcommand{\lintransf}{b}
\newcommand{\orthant}{\mathcal{O}_{\sml\leq}}

\newcommand{\proj}{s}
\newcommand{\Proj}{S}
\newcommand{\Projexpl}{\diag\{\mathds{1}_{m_i} \!:\, i = 1, \ldots, n\}}

\newcommand{\maps}{\text{Map}}

\newcommand{\actAdvsym}{\ell}
\newcommand{\actAdv}[1]{\actAdvsym_{#1}}
\newcommand{\advsym}{\text{adv}}
\newcommand{\adv}[2]{\advsym_{#1, #2}}

\newcommand{\B}{B}

\newcommand{\twonorm}[1]{\norm{#1}_2}
\newcommand{\infnorm}[1]{\norm{#1}_\infty}

\newcommand{\eqrel}{\sim_{\sml \!\actAdvsym}}

\newcommand{\mdp}{\mathfrak{M}}
\newcommand{\mdparg}[2]{\mdp\left(#1, #2\right)}

\newcommand{\normset}{\mathcal{N}}
\newcommand{\outmap}{h_{\sml \text{max}}}
\newcommand{\bellman}{\mathcal{T}}
\newcommand{\threshold}{\varepsilon_{\sml \text{min}}}

\newcommand{\prob}{\mathbb{P}}
\newcommand{\E}[2]{\mathbb{E}_{#1}\left[#2\right]}
\newcommand{\Bapx}{\hat{B}}
\newcommand{\Costpert}{\hat{\Cost}}

\renewcommand{\xmapsto}[1]{\overset{#1}{\longmapsto}}

\newtheorem{theorem}{Theorem}[]
\newtheorem{definition}[theorem]{Definition}

\newtheorem{lemma}[theorem]{Lemma}
\newtheorem{proposition}[theorem]{Proposition}
\newtheorem{corollary}[theorem]{Corollary}

\newtheorem{remark}[theorem]{Remark}

\titleformat{\section}{\normalfont\Large\centering\scshape}
  {\centering\thesection.}
  {1em}
  {}
\titleformat{\subsection}
  {\normalfont\large\centering\scshape}
  {\thesubsection.}
  {1em}
  {}

\hypersetup{ colorlinks=true, linkcolor=black, filecolor=black, urlcolor=black }

\begin{document}
\title{ON REWARD-BALANCING METHODS FOR REINFORCEMENT LEARNING} %
\setlength{\parindent}{0pt}
\author{Simone Baroncini} %
\address{%
\noindent Department of Electrical, Electronic, and Information Engineering ``Guglielmo Marconi'' - DEI\\
University of Bologna\\
Bologna, Italy}
\email{s.baroncini@unibo.it}
\author{Bahman Gharesifard} %
\address{Department of Mathematics and Statistics\\
Queen's University\\
Kingston, ON, Canada}
\email{bahman.gharesifard@queensu.ca}
\urladdr{https://gharesifard.github.io}
\author{Giuseppe Notarstefano} %
\address{Department of Electrical, Electronic, and Information Engineering ``Guglielmo Marconi'' - DEI\\
University of Bologna\\
Bologna, Italy}
\email{giuseppe.notarstefano@unibo.it}

\date{\today}

\begin{abstract}
This paper investigates the so-called reward-balancing methods, a novel class of algorithms for solving discounted-return reinforcement learning (RL) problems. These methods consist of iteratively adjusting the reward function to transform the RL problem into an equivalent one in which the optimal policies are greedy. For this procedure, referred to as normalization process, we provide a theoretical analysis of the involved transformations, emphasizing their algebraic structure. Then, we introduce a control-theoretic reformulation, recasting the reward-balancing procedure into an optimal control framework. The approach is further extended to address model uncertainty through stochastic model sampling, yielding normalization guarantees and probabilistic bounds on stochastic fluctuations. Using the proposed optimal control framework within a scenario model predictive control (MPC) setting, we demonstrate, through simulation studies, performance improvements over the current state-of-the-art.
\end{abstract} %
\maketitle

\section{Introduction}
\label{sec:introduction}
\noindent Most reinforcement learning algorithms developed within the Markov Decision Processes (MDPs) framework rely on a reward function that assigns values to the actions available in each state. Specifically, the role of rewards is to compare actions in the long term, either in terms of average or discounted return. Consequently, many MDPs can induce similar, if not identical, pairwise action comparisons, yet exhibit different performance under a given algorithm. Therefore, the task can be viewed from a different perspective, i.e., modify the reward function in a way to either gain better performance or simplify the problem. This idea has been explored in the literature in different contexts. The concept of modifying rewards without altering the set of optimal policies was first explored by Ng et al.~\cite{ng1999policy}, where potential functions were used to define policy-invariant reward transformations. This method, commonly referred to as \emph{reward shaping}, has then become a foundational tool for improving sample efficiency and convergence speed in reinforcement learning. Subsequent works have extended these ideas to more sophisticated shaping functions and learning-based shaping mechanisms, including techniques for automatic reward-balancing \cite{marthi2007automatic, ren2021orientation} and methods that learn
shaping potentials \cite{zou2019reward, hu2020learning}. The reward transformations treated in this paper were recently introduced by Mustafin and collaborators~\cite{mustafin2025mdpgeometrynormalizationreward}, where they are termed \emph{reward-balancing methods}. Unlike traditional RL algorithms, which iteratively improve the policy while keeping the reward function fixed, reward-balancing methods take an alternative approach, called \emph{normalization}, which consists of fixing the policy as greedy with respect to the current reward function and adjusting the latter to make that policy optimal. Interest in reward balancing is driven by the fact that the method proposed in \cite[Section 8]{mustafin2025mdpgeometrynormalizationreward} matches the state-of-the-art sample complexity of $Q$-learning, while being
fully parallelizable, thus representing a significant improvement over the current state-of-the-art~\cite{woo2025blessing}.

\subsubsection*{Statement of Contributions.} The main contribution of this paper is threefold. First, we provide a comprehensive theoretical analysis of the group of transformations underlying reward-balancing methods, which allows for the iterative reward adjustment process not to alter the original problem. Specifically, we study the action of this group on the space of reward functions, providing insight into the mathematical foundations of this approach, with particular focus on its underlying geometry. Second, we introduce a control-theoretic reformulation, which offers a structured and systematic approach for the design and analysis of reward-balancing methods. From this perspective, the problem is recast as an optimal control problem, which facilitates the design of novel methods through the introduction of constraints and penalty terms. Third, we extend the framework to settings with model uncertainty, adopting a stochastic model-sampling paradigm. In this context, we establish rigorous normalization guarantees and derive probabilistic bounds on the deviation from the average update, offering a quantitative measure of the stochastic fluctuations. These results ensure that reward-balancing methods remain well-posed and stable under stochastic model sampling. Finally, simulation results demonstrate that the proposed optimal control framework, implemented in a scenario receding-horizon setting, outperforms the current state-of-the-art method in terms of policy optimality.
\subsubsection*{Organization.}
The paper is organized as follows. In Section~\ref{sec:notations_and_preliminaries}, we summarize some notation and mathematical preliminaries we need to present the problem under study. In Section~\ref{sec:framework}, we introduce the discounted-return reinforcement learning framework. In Section~\ref{sec:section_1}, we introduce reward-balancing methods in the framework where the exact MDP model is known, providing their characterization as a group and studying their action on the space of reward functions. In Section~\ref{sec:section_2}, we provide a control-theoretic reformulation of these methods. Finally, in Section~\ref{sec:section_3}, we extend the framework to handle the case of uncertain model, and we provide tools for the design of robust reward-balancing laws.
\section{Mathematical Notation and Preliminaries}
\label{sec:notations_and_preliminaries}
We introduce some of the notations used, starting with some categorical ones which turn out to be useful for out purposes.
We denote by~$\maps(A_1, A_2)$ the collection of all the morphisms between two objects $A_1$ and $A_2$ in a given category. For instance, if $A_1$ and $A_2$ are topological spaces, then $\maps(A_1, A_2)$ denotes the set of all continuous maps from $A_1$ to $A_2$. The notation $A_1 \simeq A_2$, where $A_1$ and $A_2$ belong to the same category, indicates that $A_1$ and $A_2$ are isomorphic. The symbol $\id{A}$ denotes the identity map on a set $A$. We will use the symbol~$\leq$ to denote both the standard order of $\R$ and the point-wise (partial) order between real-valued functions. Specifically, if $\varphi, \psi \in \maps(A, \R)$, for some set $A$, we write $\varphi \leq \psi$ to indicate that $\varphi(p) \leq \psi(p)$ for all $p \in A$. The same convention holds for~$\geq$. Given an $n$-dimensional vector space $V$ with an ordered basis $\mathcal{B} = \{e_1, \hdots, e_n\}$ and a vector $v \in V$, we will use $v^i$ to denote the $i$-th component of $v$ with respect to the basis $\mathcal{B}$. For vectors, the symbols $\leq$ and $\geq$ are also used to denote the component-wise (partial) order, whenever a basis is fixed. For Euclidean spaces, we use the symbol $\|\cdot\|_{p}$ to denote the $p$-norm. If we do not refer to a specific norm, we instead use the general symbol $\|\cdot\|$. We denote by $\hat{e}_i \in \R^n$ the canonical $i$-th basis element, i.e. $\hat{e}_i^j = 1$ if $j = i$, $\hat{e}_i^j = 0$ otherwise. If $M$ is a matrix, $M_j^i$ denotes the element in the $i$-th row and $j$-th column. The identity matrix of dimension $n$ is denoted by~$I_n$, while $\mathds{1}_n \in \R^n$ denotes the vector of all ones. The notation $\diag\{M_1, \hdots, M_n\}$ refers to the block-diagonal matrix whose $i$-th block is $M_i$. If $X$ is a random variable and $Y$ is a set of conditions on $X$, $\E{}{X | Y}$ denotes the expectation of $X$ given $Y$. If $E$ is an event in a probability space, we denote by $\prob(E)$ the corresponding probability.

Throughout the paper, we assume the basic knowledge of groups. Nevertheless, we briefly review some key concepts that are useful for understanding some of the main results of the paper.
Recall that a \emph{homomorphism} between two groups $(G, *), (H, \cdot)$ is a map $\psi : G \to H$ satisfying
\begin{align*}
    \psi(g * h) = \psi(g) \cdot \psi(h),
\end{align*}
for all $g, h \in G$, whereas an \emph{anti-homomorphism} inverts the order of multiplication $\psi(g * h) = \psi(h) \cdot \psi(g)$. The symmetric group $\text{Sym}(S)$ of a set $S$ is the group of bijections from $S$ to itself, in which composition is group multiplication. Given a group $(G, *)$ and a set $S$, a left action of $G$ on $S$ is a homomorphism from $(G, *)$ to the symmetric group of $S$, namely a map $a : G \to \text{Sym}(S)$ such that
\begin{align*}
    a(g * h) = a(g) \circ a(h).
\end{align*}
Equivalently, it is a map
\begin{align*}
    \beta : G \times S &\to S\\
    (g, s) &\mapsto a_g(s) \coloneqq a(g)(s).
\end{align*}
A right action is an anti-homomorphism from $(G, *)$ to the symmetric group of $S$. The orbit $O_s$ of an element $s \in S$ is the subset of $S$ reachable from $s$ under the action of $(G, *)$:
\begin{align*}
    O_s \coloneqq \{a_g(s) : g \in G\}.
\end{align*}
Given two elements $s_1, s_2$, we write $s_1 \sim_a s_2$ if they belong to the same orbit, i.e. if there exists $g \in G$ such that $s_2 = a_g(s_1)$. The relation $\sim_a$ is an equivalence relation and the orbit $O_s$ is an equivalence class thereof. With slight abuse of notation, we will often denote a group $(G, *)$ simply by $G$.

Furthermore, we provide a brief overview of some fundamental concepts of fiber bundles, which are frequently used throughout the text. A fiber bundle is a tuple $(M, N, \mathcal{P}, F)$ where $M$, $N$ and $F$ are topological spaces, called total space, base space, and model fiber, respectively, and $\mathcal{P} : M \to N$ is a surjective continuous map, with the property that, for every point $p \in N$ there exists a neighborhood $U$ of $p$ whose inverse image $\mathcal{P}^{-1}(U)$ is homeomorphic to the product $U \times F$. With slight abuse of notation, we call the inverse image of any point $p \in U$ the fiber of $p$. The homeomorphism $\psi : \mathcal{P}^{-1}(U) \to U \times F$ is called a \emph{local trivialization} and satisfies $\mathcal{P} = \mathcal{P}_1 \circ \psi$, where $\mathcal{P}_1 : U \times F \to U$ is the projection onto the first factor. If the set $U$ can be chosen to be the whole base space, then the total space is homeomorphic to the product $N \times F$ and is called a \emph{trivial bundle}. Given a topological group $G$ (i.e., a group equipped with a topology) with multiplication $*$, a \emph{principal $G$-bundle} is a fiber bundle whose fiber model is $G$ and whose total space is acted upon by $G$ through a continuous, fiber-preserving right action $a$ such that:
\begin{enumerate}[label=$\rhd$]
    \item the action is transitive on each fiber, i.e., for any $p \in M$ and any pair of points $q_1, q_2 \in \mathcal{P}^{-1}(p)$, there exists a group element $g \in G$ such that $a_g(q_1) = q_2$,
    \item the action is free, i.e., no nontrivial element of the group fixes a point,
    \item local trivializations are $G$-equivariant, i.e., if $\psi : \mathcal{P}^{-1}(U) \to U \times G$ is a local trivialization, then:
    \begin{align*}
        \psi(q) = (\mathcal{P}(q), g_1) \implies \psi(a_{g_2}(q)) = (\mathcal{P}(q), g_1 * g_2).
    \end{align*}
\end{enumerate}
In particular, each fiber $\mathcal{P}^{-1}(p)$ is precisely one orbit and is homeomorphic to $G$. %
\section{Reinforcement Learning Framework}
\label{sec:framework}
In this paper, we focus on classes of Markov decision processes sharing, up to a bijection, a finite state space $\X$ of cardinality $n$. We define the action space\footnote{In the literature, this is usually referred to as ``state-action space'', i.e., the collection of all pairs consisting of a state and an action available in that state.} $\U$ as the disjoint union
\begin{align*}
    \U \coloneqq \bigsqcup_{\state \in \X} \widetilde{\U}_{\state},
\end{align*}
where $\widetilde{\U}_{\state} \subseteq \widetilde{\U}$, with finite cardinality $m_{\state}$, denotes the set of all actions from an action pool $\widetilde{\U}$ which can be taken from the state $\state$. It follows that the action space has cardinality $m \coloneqq \sum_\state m_\state$. This construction admits a natural projection map $\proj : \U \to \X$, which maps each ``disjoint'' action to the state it belongs to. Henceforth, we will denote by $\U_\state$ the fiber $\proj^{-1}(\state)$ of the state $\state$. In this framework, a deterministic stationary policy (DSP or, simply, policy) $\policy$ can be conveniently defined as a \emph{section} of the projection map, i.e., a map $\policy : \X \to \U$ such that $\proj \circ \policy = \id{\X}$. We denote the set of all the deterministic stationary policies by $\Sect$. For convenience, we define a preliminary classification of MDPs based solely on the ``bundle'' construction introduced so far. 
\begin{definition}
    We say that two MDPs are comparable iff there exists a bijective bundle map\footnote{Given two projections $\proj_1: \U_1 \to \X_1$ and $\proj_2: \U_2 \to \X_2$, a map $F : \U_1 \to \U_2$ is said to be a \textit{bundle map} iff it is fiber-preserving, i.e., there exists a map $G : \X_1 \to \X_2$ such that $\proj_2 \circ F = G \circ \proj_1$.} between their action spaces.
\end{definition}
\noindent
Within each class, every MDP is characterized by its reward function $\cost : \U \to \R$, defining the goodness of each action, and its probability transition function $\f : \X \times \U \to [0, 1]$, where $\f(\state^+\!, \act)$ denotes the probability of transitioning to $\state^+$ after applying $\act$ (from the state $\proj(\act)$). A policy $\policy$ naturally induces pullbacks of both functions onto the state space given by
\begin{subequations}\label{eq:pullback_functions}
    \begin{align}
        \costfb{\policy} &\coloneqq \cost \circ \policy,\\
        \ffb{\policy} &\coloneqq \f \circ (\id{\X} \times \policy).
    \end{align}
\end{subequations}
The objective of discounted-return reinforcement learning is to find a DSP maximizing the so-called value function $\vv : \X \to \R$ which assigns to each state the expected long-run cumulative reward associated with the given policy, and is formally defined as
\begin{align}\label{eq:value_funct_def}
    \vv(\state) = \E{}{\sum_{t = 0}^{\infty} \disc^t \costfb{\policy}(\state_t) \biggr\rvert 
    \begin{array}{l}
      \state_0 = \state\\
      \state_{t + 1} \sim \ffb{\policy}(\cdot, \state_t)
   \end{array}
   },
\end{align}
for some discount factor $\disc \in [0, 1)$. Specifically, we aim to find $\policy^* \in \Sect$ such that:
\begin{align*}
    \vvarg{\policy^*\!,\cost}(\state) \geq \vv(\state), \quad \forall \state \in \X, \forall \policy \in \Sect.
\end{align*}
Since in this paper we will only consider comparable MDPs, we will denote an MDP with reward $\cost$, transition function $\f$ and discount factor $\disc$ as $\mdparg{\cost}{\disc \f}$. Note that the knowledge of $\disc \f$ is enough to retrieve both $\disc$ and $\f$. An optimal DSP is always guaranteed to exist in this framework (see, e.g., \cite[Theorem 7.1.9]{puterman2014markov}), making the problem well-posed. Moreover, the value function $\vv$ can be characterized as the unique solution of the renowned Bellman equation
\begin{align}\label{eq:bellman_equation_v}
    v(\state) = \costfb{\policy}(\state) + \disc  \sum_{z \in \X} \ffb{\policy}(z, \state) v(z), \quad \forall\state \in \X.
\end{align}
This condition, can be also expressed in terms of action-level comparisons by introducing the advantage function $\adv{\policy}{\cost} : \U \to \R$ corresponding to the policy $\policy$, defined by
\begin{align*}
    \adv{\policy}{\cost}(\act) \coloneqq \cost(\act) + \disc  \sum_{z \in \X} \f(z, \act) \vv(z) - \vv(\proj(\act)),
\end{align*}
for all $\act \in \U$. As the name suggests, this function quantifies the benefit of deviating from the current policy by selecting $\act$ in place of the action $\policy(\state)$ taken by the policy in state $\state = \proj(\act)$. Intuitively, if $\adv{\policy}{\cost}(\act) > 0$ then replacing $\policy(\state)$ with $\act$ improves the expected return, whereas if $\adv{\policy}{\cost}(\act) < 0$ then keeping the current policy is preferable. With this formulation, the fact that $\vv$ satisfies the equation \eqref{eq:bellman_equation_v} can be compactly written as 
\begin{align}\label{eq:bellman_equation_v_adv}
    \adv{\policy}{\cost} \circ \policy = 0.
\end{align}
The advantage function also allows us to introduce the action-value function $\q : \U \to \R$, commonly referred to as $Q$-function, defined as
\begin{align}
    \q \coloneqq \adv{\policy}{\cost} + \vv \circ \proj,
\end{align}
which assigns to each action a long-run expected return just as the value function does to states. The $Q$-function is in turn the unique solution of an action-level version of \eqref{eq:bellman_equation_v}:
\begin{align}\label{eq:bellman_q_fun}
    \q(\act) &= \cost(\act) + \disc \sum_{z \in \X} \f(z, \act) \bigl[\q \circ \policy\bigr](z), \quad \forall \act \in \U.
\end{align}
Furthermore, by \eqref{eq:bellman_equation_v_adv}, it follows that the value function is given by the pullback, via the policy $\policy$, of the $Q$-function:
\begin{align*}
    \vv = \q \circ \policy.
\end{align*}
\noindent
We then introduce two meaningful relations that enable us to further partition a class of comparable MDPs.
\begin{definition}[\textbf{Reward-coupled and model-coupled MDPs}]
    We say that two comparable MDPs are:
    \begin{enumerate}[label=$\rhd$]
        \item reward-coupled iff their reward functions coincide, up to a bijection of action spaces,
        \item model-coupled iff they have the same discount factor and their transition functions coincide, up to a bijection of state and action spaces.
    \end{enumerate}
\end{definition}
\begin{figure}[t]
    \centering
    \includegraphics[width = 0.6\textwidth]{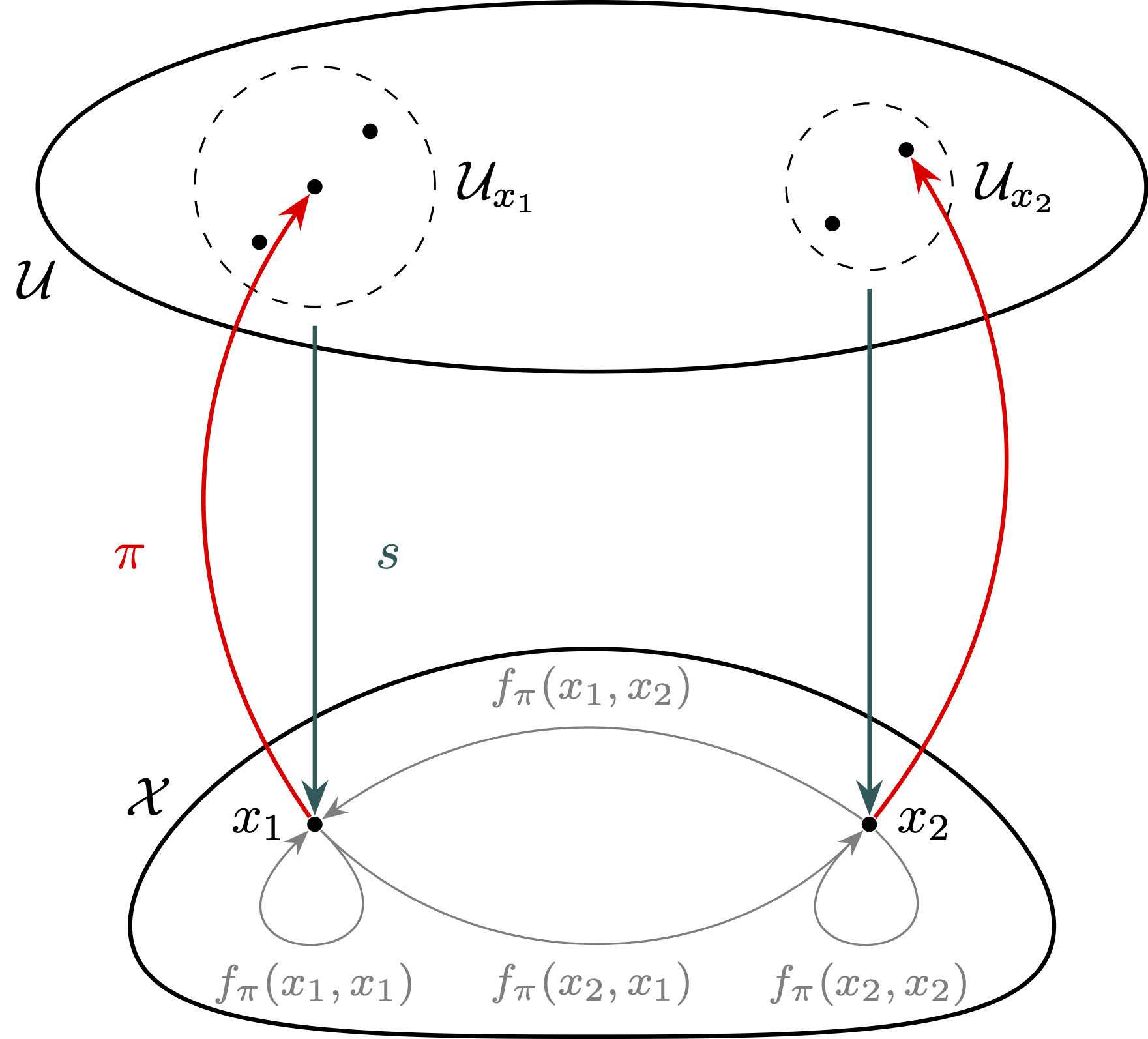}
    \caption{Representation of an MDP with two states $x_1, x_2$ and five actions, three available in $x_1$ and two available in $x_2$.}
    \label{fig:mdp_bundle}
\end{figure}
\noindent
It is often convenient to identify the functions introduced so far with proper vectors and matrices. In order to do so, we index both the state space $\X$, by identifying it with the set of the first $n$ positive integers, and, accordingly, the action space $\U$, with the set of the first $m$ positive integers, by first enumerating the actions in the fiber of the first state, then those in the fiber of the second state, and so on. Let us denote by $\IdxX : \X \to \{1, \hdots, n\}$ the first index map and by $\IdxU : \U \to \{1, \hdots, m\}$ the second one. With this indexing in mind, in order not to burden the notation any further we will use in the following $m_i$ in place of $m_{\state}$ and $\U_i$ in place of $\IdxU(\U_\state)$ whenever $i = \IdxX(\state)$.
We can then identify the reward function $\cost$ with a vector $\Cost \in \R^m$ whose components satisfy
\begin{align*}
    \Cost^{i} = \bigl(\cost \circ \IdxU^{-1}\bigr)(i),
\end{align*}
for all $i \in \{1, \hdots, m\}$, and the probability transition function $\f$ with a row-stochastic matrix $\F \in [0, 1]^{m \times n}$, whose entries satisfy
\begin{align*}
    \F^i_j = \bigl[\f \circ (\IdxX \times \IdxU)^{-1}\bigr](j, i),
\end{align*}
for all $i \in \{1, \hdots, m\}$ and $j \in \{1, \hdots, n\}$. Furthermore, a DSP $\policy$ can be represented by a matrix $\Pi \in \{0, 1\}^{n\times m}$ satisfying $\Pi_j^i = 1$ if $j = (\IdxU \circ \policy \circ \IdxX^{-1})(i)$ and $\Pi_j^i = 0$, otherwise. In this way the pullbacks representations $\Costfb{\policy} \in \R^n$ and $\Ffb{\policy} \in \R^{n\times n}$ can be easily computed as $\Costfb{\policy} = \Pi \Cost$ and $\Ffb{\policy} = \Pi \F$, respectively. Consequently, the projection map $\proj$ is represented by the matrix
\begin{align}\label{eq:proj_mat}
    \Proj \coloneqq \Projexpl.
\end{align}
Observe that $\Pi \Proj = I_n$, for any policy matrix $\Pi$.
Eventually, as done for the other functions, we can represent both value functions with two vectors $\V \in \R^n$ and $\Q \in \R^m$. In this representation, we can exploit the Bellman equations \eqref{eq:bellman_equation_v}-\eqref{eq:bellman_q_fun} to express the value vectors as functions of the reward vectors as follows
\begin{align}\label{eq:value_reward_relation}
    \V &= (I_n - \disc \Ffb{\policy})^{-1} \Costfb{\policy},\\
    \Q &= (I_m - \disc \F\Pi)^{-1} \Cost,
\end{align}
for all $\policy \in \Sect$. These two representations are naturally related by $\V = \Pi \Q$, as proved by the following identities:
\begin{align*}
    \Pi \Q &= \Pi (I_m - \disc \F\Pi)^{-1} \Cost\\
    &= \sum_{t = 0}^\infty \disc^t \Pi(\F\Pi)^t \Cost\\
    &= \sum_{t = 0}^\infty \disc^t (\Pi\F)^t \Pi\Cost\\
    &= (I_n - \disc \Ffb{\policy})^{-1} \Costfb{\policy}\\
    &= \V.
\end{align*}
The last ingredient which we require is the notion of symmetry for MDPs, as we outline next.
\subsubsection*{MDP Symmetry.}
In many cases, the MDP of interest exhibits some form of symmetry, meaning that certain states and actions are interchangeable. For example, in the game of tic-tac-toe, an empty board with a cross in the top-right corner is equivalent, from the point of view of the game dynamics, to an empty board with the cross in any other corner, provided that the available actions are rotated accordingly. More generally, a symmetry in an MDP implies that there exists a group of transformations acting simultaneously on the state and action spaces, under which the transition and reward functions remain invariant. We formalize this notion in the following definition (see, e.g., \cite{van2020mdp}).
\begin{definition}[\textbf{Group-invariant MDPs}]\label{def:group_invariant_mdp}
    Let $\mdparg{\cost}{\disc \f}$ be an MDP and let $G$ be a group acting on both $\X$ and $\U$ through left actions $\reprarg{\mathrm{x}}{}$ and $\reprarg{\mathrm{u}}{}$, respectively. We say that the tuple $\mdp_G(\cost, \disc \f) \coloneqq (\mdparg{\cost}{\disc \f}, G)$ is a $G$-invariant MDP iff the following properties hold:
    \begin{align*}
        \text{(Reward Invariance): }& \cost\left(\reprarg{\mathrm{u}}{g}(\act)\right) = \cost(\act)\\
        \text{(Transition Invariance): }& \f\left( \reprarg{\mathrm{x}}{g}(\state^+),~ \reprarg{\mathrm{u}}{g}(\act) \right) = \f(\state^+, \act)
    \end{align*}
    for all $\state^+ \in \X$, $\act \in \U$ and $g \in G$.
\end{definition}
\noindent
The group actions $\reprarg{\mathrm{x}}{}$ and $\reprarg{\mathrm{u}}{}$ can be viewed as homomorphisms from the group $G$ to the symmetric groups $\text{Sym}(\X) \simeq S_n$ and $\text{Sym}(\U) \simeq S_m$, respectively. Therefore, in the finite-dimensional vector representation previously introduced, $\reprarg{\mathrm{x}}{g}$ and $\reprarg{\mathrm{u}}{g}$ are represented by two permutation matrices $A_g^{\mathrm{x}} \in \{0, 1\}^{n\times n}$ and $A_g^{\mathrm{u}} \in \{0, 1\}^{m\times m}$. Hence, in this representation, the group invariant conditions of Definition \ref{def:group_invariant_mdp} read as:
\begin{subequations}\label{eq:group_invariant_mdp_vector_form}
    \begin{align}
        &A_g^{\mathrm{u}} \Cost = \Cost\\
        &A_g^{\mathrm{u}} \F \left(A_g^{\mathrm{x}}\right)\T = \F.
    \end{align}
\end{subequations}
To illustrate the benefit of exploiting MDP symmetries, consider a simple example: a ``cyclic corridor'' with three states $\{x_1, x_2, x_3\}$ and two actions a-priori: $u_l$ to move left ($\text{mod } 3$) and $u_r$ to move right ($\text{mod } 3$). The action space is made of three copies of these two actions. We will denote by $u_l^i$ and $u_r^i$ the left and right actions taken from state $x_i$. Hence, the action space $\U$ has cardinality 6. The transition function can be represented as follows:
\begin{center}
    \begin{tabular}[]{c|ccc}
        $f$ & $x_1$ & $x_2$ & $x_3$\\
        \hline
        $u_l^1$ & 0 & 0 & 1\\
        $u_r^1$ & 0 & 1 & 0\\
        $u_l^2$ & 1 & 0 & 0\\
        $u_r^2$ & 0 & 0 & 1\\
        $u_l^3$ & 0 & 1 & 0\\
        $u_r^3$ & 1 & 0 & 0\\
    \end{tabular}
\end{center}
Suppose that our goal is to make the agent frequently visit the state $x_2$ (i.e., it should keep returning to $x_2$). To this end, we may define the following reward map:
\begin{center}
    \begin{tabular}[]{c|cc}
        $\cost$ & $u_l$ & $u_r$\\
        \hline
        $x_1$ & -1 & 0\\
        $x_2$ & -1 & -1\\
        $x_3$ & 0 & -1\\
    \end{tabular}
\end{center}
Now, consider the group $S_2 \coloneqq \{e, \sigma\}$ of permutations of 2 elements acting on this MDP. The group acts on $\tilde{\U}$ by either doing nothing ($\reprarg{\tilde{\mathrm{u}}}{e}$) or by swapping the two actions ($\reprarg{\tilde{\mathrm{u}}}{\sigma}$), and it acts on the state space by either doing nothing ($\reprarg{\mathrm{x}}{e}$) or swapping $x_1$ and $x_3$ ($\reprarg{\mathrm{x}}{\sigma}$). Specifically, the action of $G$ on $\X$ defines an injective group homomorphism from $G$ to the symmetric group of degree 3, $S_3$, whose range is the subgroup generated by the 2-cycle $\sigma_{13} \coloneqq (13)$. The resulting non-trivial action on $\U$ consists of applying $\reprarg{\tilde{\mathrm{u}}}{\sigma}$ to each fiber, followed by the permutation of the fibers induced by $\reprarg{\mathrm{x}}{\sigma}$. With respect to the actions of the group generator $\sigma$, the MDP exhibits transition invariance
\begin{center}
    \scalebox{1}{
    \begin{tabular}[]{c|ccc}
        $f$ & $x_1$ & $x_2$ & $x_3$\\
        \hline
        $u_l^1$ & 0 & 0 & 1\\[2pt]
        $u_r^1$ & 0 & 1 & 0\\[2pt]
        $u_l^2$ & 1 & 0 & 0\\[2pt]
        $u_r^2$ & 0 & 0 & 1\\[2pt]
        $u_l^3$ & 0 & 1 & 0\\[2pt]
        $u_r^3$ & 1 & 0 & 0\\[2pt]
    \end{tabular}
    $~\xmapsto{\text{$\reprarg{\tilde{\mathrm{u}}}{\sigma}$}}~$
    \begin{tabular}[]{c|ccc}
        $f$ & $x_1$ & $x_2$ & $x_3$\\
        \hline
        $u_r^1$ & 0 & 1 & 0\\[2pt]
        $u_l^1$ & 0 & 0 & 1\\[2pt]
        $u_r^2$ & 0 & 0 & 1\\[2pt]
        $u_l^2$ & 1 & 0 & 0\\[2pt]
        $u_r^3$ & 1 & 0 & 0\\[2pt]
        $u_l^3$ & 0 & 1 & 0\\[2pt]
    \end{tabular}
    $~\xmapsto{\text{$\reprarg{\mathrm{x}}{\sigma}$}}~$
    \begin{tabular}[]{c|ccc}
        $f$ & $x_3$ & $x_2$ & $x_1$\\
        \hline
        $u_r^3$ & 0 & 0 & 1\\[2pt]
        $u_l^3$ & 0 & 1 & 0\\[2pt]
        $u_r^2$ & 1 & 0 & 0\\[2pt]
        $u_l^2$ & 0 & 0 & 1\\[2pt]
        $u_r^1$ & 0 & 1 & 0\\[2pt]
        $u_l^1$ & 1 & 0 & 0\\[2pt]
    \end{tabular}
    }
\end{center}
and reward invariance
\begin{center}
    \scalebox{1}{
    \begin{tabular}[]{c|cc}
        $\cost$ & $u_l$ & $u_r$\\
        \hline
        $x_1$ & -1 & 0\\
        $x_2$ & -1 & -1\\
        $x_3$ & 0 & -1
    \end{tabular}
    $~\xmapsto{\text{$\reprarg{\tilde{\mathrm{u}}}{\sigma}$}}~$
    \begin{tabular}[]{c|cc}
        $\cost$ & $u_r$ & $u_l$\\
        \hline
        $x_1$ & 0 & -1\\
        $x_2$ & -1 & -1\\
        $x_3$ & -1 & 0
    \end{tabular}
    $~\xmapsto{\text{$\reprarg{\mathrm{x}}{\sigma}$}}~$
    \begin{tabular}[]{c|cc}
        $\cost$ & $u_r$ & $u_l$\\
        \hline
        $x_3$ & -1 & 0\\
        $x_2$ & -1 & -1\\
        $x_1$ & 0 & -1
    \end{tabular}
    }
\end{center}
This implies that the considered MDP is $S_2$-invariant. This allows us to identify the original MDP with a ``reduced'' one having:
\begin{enumerate}[label=$\rhd$]
    \item a reduced state space:
    \begin{align*}
        \X' \coloneqq \{[x_1], x_2\}, \quad [x_1] \coloneqq \{x_1, x_3\},
    \end{align*}
    \item a reduced action space:
    \begin{align*}
        \U' \coloneqq \{[u_l^1], [u_r^1], [u_l^2]\},
    \end{align*}
    where
    \begin{align*}
        [u_l^1] = \{u_l^1, u_r^3\}, \quad [u_r^1] = \{u_r^1, u_l^3\}, \quad [u_l^2] = \{u_l^2, u_r^2\},
    \end{align*}
    \item a reduced transition map:
    \begin{center}
        \begin{tabular}[]{c|cc}
            $f'$ & $[x_1]$ & $x_2$\\
            \hline
            $[u_l^1]$ & 1 & 0\\
            $[u_r^1]$ & 0 & 1\\
            $[u_l^2]$ & 1 & 0
        \end{tabular}
    \end{center}
    \item a reduced reward map:
    \begin{center}
        \begin{tabular}[]{c|cc}
            $\cost$ & $u_l$ & $u_r$\\
            \hline
            $[x_1]$ & -1 & 0\\
            $x_2$ & -1 & 
        \end{tabular}
    \end{center}
\end{enumerate}
\begin{center}
    \begin{tikzpicture}[use Hobby shortcut, scale = 1]
        \tikzset{sphere/.style={
        draw,
        thick,
        #1!75!black,
        circle,
    }}
        \tikzset{arrow/.style={
            very thick,
            ->,
            >=latex
        }}
        \begin{scope}
            \node[sphere = black] (p1) at (0.0, 0.0) {$x_1$};
            \node[sphere = black] (p2) at (-1.5, -2) {$x_2$};
            \node[sphere = black] (p3) at (1.5, -2) {$x_3$};
        
            \draw[-stealth, red](p1) to[bend right=20] node[above, midway, rotate = 53.13, scale = 0.5] {$(u_r^1, 0)$} (p2);
            \draw[-stealth, blue](p2) to[bend right=20] node[below, midway, rotate = 53.13, scale = 0.5] {$(u_l^2, -1)$} (p1);
            \draw[-stealth, brown](p1) to[bend right=20] node[below, midway, rotate = -53.13, scale = 0.5] {$(u_l^1, -1)$} (p3);
            \draw[-stealth, brown](p3) to[bend right=20] node[above, midway, rotate = -53.13, scale = 0.5] {$(u_r^3, -1)$} (p1);
            \draw[-stealth, blue](p2) to[bend right=20] node[below, midway, scale = 0.5] {$(u_r^2, -1)$} (p3);
            \draw[-stealth, red](p3) to[bend right=20] node[above, midway, scale = 0.5] {$(u_l^3, 0)$} (p2);
        \end{scope}
        \node[] () at (2.5, -2.0) {$\xrightarrow{\sim_G}$};
        \begin{scope}[xshift=5cm]
            \node[sphere = black] (p2) at (-1.5, -2) {$x_2$};
            \node[sphere = black, scale = 0.8] (p3) at (1.5, -2) {$[x_1]$};

            \draw[-stealth, blue](p2) to[bend right=20] node[below, midway, scale = 0.5] {$([u_l^2], -1)$} (p3);
            \draw[-stealth, red](p3) to[bend right=20] node[above, midway, scale = 0.5] {$([u_r^1], 0)$} (p2);

            \draw[-stealth, brown](p3.north east) arc [start angle=-50, end angle=220, radius=0.4] node[above, midway, scale = 0.5] {$([u_l^1], -1)$};
        \end{scope}
    \end{tikzpicture}
\end{center}
This example shows how exploiting the symmetries of the underlying MDP can help reduce its size, and thus the complexity of the reinforcement learning methods employed. \section{Exact-Model Reward-Balancing Methods}
\label{sec:section_1}
In this section, we focus on model-coupled MDPs and examine the reward transformations involved in reward-balancing methods. We highlight their main properties and provide the mathematical foundations of the control-theoretic framework developed later in the paper.
\subsection{Problem Statement}
In \cite{mustafin2025mdpgeometrynormalizationreward}, the authors propose a class of methods for modifying the reward function of an MDP to transform it into a \emph{normal MDP}. We recall this definition.
\begin{definition}[\textbf{Normal MDPs}]
    We say that an MDP is in normal form iff its optimal value function is zero.
\end{definition} 
\noindent
Relation \eqref{eq:value_reward_relation} implies that, in a normal MDP, every state has (at least) one action with zero reward in its fiber, which lies in the image of an optimal policy. Thus, the key advantage of these MDPs is that to find an optimal policy it is sufficient to select, for each state, an action with zero reward. Now, the challenge is to ``normalize'' a given MDP without altering the problem. That is, the transformation should preserve the relative position of the value functions, which induce a preorder over the policy space. We now state the problem addressed in this paper:
\begin{enumerate}[label=(\alph*)]
    \item Given a class of MDPs sharing the same \emph{known} model:
    \begin{enumerate}[label = $\rhd$]
        \item characterize the set of normal MDPs and its relationship to the reward transformations involved in reward-balancing methods (cf. Section~\ref{subsec:section_1A}), and
        \item identify a subset of these transformations that ``improve'' the reward function (cf. Section~\ref{subsec:section_1B});
    \end{enumerate}  
    \item Cast the normalization procedure as an optimal control problem in which the state and the input are the reward function and the transformation, respectively (cf. Section~\ref{sec:section_2});
    \item In realistic settings where the MDP is \emph{unknown} but can be sampled, leverage this control-theoretic formulation to design reward-balancing strategies outperforming the existing one in the literature (cf. Section~\ref{sec:section_3}).
\end{enumerate}
\subsection{Advantage-Preserving Reward Transformations}
\label{subsec:section_1A}
The set of reward transformations introduced in \cite{mustafin2025mdpgeometrynormalizationreward} is a subgroup of the affine group of $\maps(\U, \R) \simeq \R^m$, consisting of all transformations of the form
\begin{align}\label{eq:reward_transform}
    \actAdv{\delta}: \cost \mapsto \cost + \lintransf(\delta),
\end{align}
parametrized by the real-valued functions $\delta : \X \to \R$ on the state space through the injective linear operator $\lintransf : \maps(\X, \R) \to \maps(\U, \R)$ defined by 
\begin{align*}
    \lintransf(\delta)(\act) = \delta(\proj(\act))\bigl(\disc \f(\proj(\act), \act) - 1\bigr) + \disc\!\!\! \sum_{y \neq \proj(\act)}\!\! \delta(y) \f(y, \act),
\end{align*}
where $\act \in \U$. Denoting by $\Delta \in \R^n$ the vector representation of $\delta$, according to the notation introduced in the previous section, we can write the reward update as
\begin{align}\label{eq:reward_transform_representation}
    \mathcal{L}^\Delta: \Cost \mapsto \Cost + \left( \disc \F - \Proj \right) \Delta,
\end{align}
where $\Proj$ is defined as in \eqref{eq:proj_mat}. It is interesting to observe that this subgroup of transformations is the result of a faithful group action of the additive group $\maps(\X, \R) \simeq \R^n$ on $\maps(\U, \R) \simeq \R^m$. More precisely, the map $\actAdv{} : \maps(\X, \R) \times \maps(\U, \R) \to \maps(\U, \R)$, satisfies:
\begin{align*}
    &\actAdv{0} = \id{\U}\\
    &\actAdv{\delta_1 + \delta_2} = \actAdv{\delta_1} \circ \actAdv{\delta_2}
\end{align*}
and, since $\lintransf$ is injective, for every distinct pair $\delta_1 \neq \delta_2$ there correspond distinct transformations $\actAdv{\delta_1} \neq \actAdv{\delta_2}$. It immediately follows that the ``reward space'' $\maps(\U, \R)$ can be partitioned by the orbits of this group action (henceforth referred to as $\actAdvsym$-orbits), which are affine $n$-dimensional subspaces. This allows us to align our treatment with \cite{mustafin2025mdpgeometrynormalizationreward} and define an equivalence relation between model-coupled MDPs.
\begin{definition}
    We consider the reward functions $\cost_1, \cost_2$ of two model-coupled MDPs to be equivalent, and we write $\cost_1 \eqrel \cost_2$, iff they belong to the same $\actAdvsym$-orbit. In symbols 
    \begin{align*}
        \cost_1 \eqrel \cost_2 ~\Longleftrightarrow~ \exists \delta \in \maps(\X, \R) : \cost_2 = \actAdv{\delta}(\cost_1).
    \end{align*}
    Two model-coupled MDPs with equivalent reward functions are said to be affine-equivalent.
\end{definition}
\begin{remark}
    For model-coupled MDPs, each MDP is uniquely determined by its reward function. Thus, in this context, we will use ``MDP'' and ``reward function'' interchangeably.
\end{remark}
\noindent
Next, we recall, without proof, the result provided in \cite{mustafin2025mdpgeometrynormalizationreward} showing that this group action preserves the advantage functions.
\begin{proposition}[{\cite[Theorem 3.5]{mustafin2025mdpgeometrynormalizationreward}}]\label{lem:advantage_preservation}
    If $\cost \eqrel \cost^\prime$ are two equivalent reward functions, then:
    \begin{align*}
        \adv{\policy}{\cost^\prime} = \adv{\policy}{\cost}, \quad \forall \policy \in \Sect.
    \end{align*}
    Moreover, the corresponding value functions are related by:
    \begin{align*}
        \vvarg{\policy, \cost^\prime} = \vv - \delta, \quad \forall \policy \in \Sect.
    \end{align*}
\end{proposition}
\noindent
In view of the previous proposition, we will refer to the transformations~\eqref{eq:reward_transform} as \emph{advantage-preserving}. Another useful property shown in \cite{mustafin2025mdpgeometrynormalizationreward} is recalled in the following lemma.
\begin{lemma}\label{lem:transform_as_advantage}
    For any reward function $\cost$ and deterministic policy $\policy$, it holds
    \vspace{-0.2cm}
    \begin{align*}
        \actAdv{\scalebox{0.7}{$\vv$}}(\cost) = \adv{\policy}{\cost}.
    \end{align*}
\end{lemma}
\noindent
In summary, if we set $\delta$ as the value function associated to a given policy and reward function, the corresponding transformation maps the reward function into its advantage function relative to that policy. The above results lead to the following corollary, which follows from the properties of the advantage function $\adv{\policy^*\!}{\cost}$ relative to an optimal policy $\policy^*$.
\begin{corollary}\label{cor:ideal_transformation}
    For any reward function $\cost$, if $\vvarg{\policy^*\!, \cost}$ is the optimal value function, then the transformed reward function $\actAdv{\scalebox{0.7}{$\vvarg{\policy^*\!, \cost}$}}\!(\cost)$ is nonpositive and the corresponding MDP is in normal form.
\end{corollary}
\noindent
Leveraging these observations, we provide the following result, which is a consequence of the faithfulness of the action~$\actAdvsym$.
\begin{theorem}[\textbf{Characterization of normal MDPs}]\label{th:normal_form_characterization}
    Consider a class $\mdparg{\cdot}{\disc \f}$ of model-coupled MDPs. Then:
    \begin{enumerate}[label=(\alph*), ref=\thetheorem-(\alph*)]
        \item\label{lem:uniqueness_of_normal_form} In every $\actAdvsym$-orbit there exists a unique normal MDP.
        \item\label{prop:nec_suff_condition_for_normality} An MDP $\mdparg{\cost}{\disc \f}$ is in normal form if and only if its reward function satisfies
        \begin{align}\label{eq:sufficient_condition_for_normality}
            \max_{\act \in \U_\state}\, \cost(\act) = 0, \quad \forall \state \in \X.
        \end{align}
    \end{enumerate}
\end{theorem}
\begin{proof}
    $(a)$ The existence follows directly from Corollary~\ref{cor:ideal_transformation}. To show uniqueness assume $\policy^*$ is an optimal policy for the MDP and assume the latter to be in normal form. In particular, by \eqref{eq:value_reward_relation}, this implies that $\costfb{\policy^*} = 0$. Now, consider any normal reward function $\cost^\prime$ in the same orbit, i.e., such that $\costfb{\policy^*}^\prime = 0$ and $\actAdv{\delta}(\cost) = \cost^\prime$ for some $\delta \in \maps(\X, \R)$. From the latter, written in vector form, we have that
    \begin{align*}
        \Costfb{\policy^*}^\prime &= \bigl( \mathcal{L}^{\Delta}\Cost \bigr)_{\policy^*}\\
        &= \bigl( \Cost + \left( \disc \F - \Proj \right) \Delta \bigr)_{\policy^*}\\
        &= \bigl( \underbrace{\Pi^*\Cost}_{= 0} + \,\Pi^*\!\left( \disc \F - \Proj \right) \Delta \bigr)\\
        &= -\bigl( I_n - \disc \Ffb{\policy^*} \bigr) \Delta
    \end{align*}
    and thus, from the normality of $\Cost^\prime$,
    \begin{align*}
        0 = \Costfb{\policy^*}^\prime = -\bigl( I_n - \disc \Ffb{\policy^*} \bigr) \Delta ~\Longleftrightarrow~ \Delta = 0,
    \end{align*}
    since $(I_n - \disc \F_{\policy^*}) \in \gln$. Hence, $\cost^\prime = \actAdv{0}(\cost) = \cost$, which proves the uniqueness.\newline $(b)$ By putting together Corollary~\ref{cor:ideal_transformation} and Theorem~\ref{lem:uniqueness_of_normal_form} it follows that the reward function of a normal MDP is nonpositive and maps to zero at least one action in the fiber of each state, implying \eqref{eq:sufficient_condition_for_normality}. Conversely, if we assume \eqref{eq:sufficient_condition_for_normality}, the reward function is nonpositive and there exists a policy $\hat{\policy} \in \Sect$ satisfying $\hat{\policy}(\state) \in \argmax_{\act \in \U_\state} \cost(\act)$, for all $\state \in \X$, such that $\vvarg{\hat{\policy}, \cost} = 0$, see \eqref{eq:value_reward_relation}. Now, by Lemma~\ref{lem:transform_as_advantage},
    \begin{align*}
        \adv{\hat{\policy}}{\cost} = \actAdv{\vvarg{\hat{\policy}, \cost}}(\cost) = \actAdv{0}(\cost) = \cost \leq 0,
    \end{align*}
    meaning that any action differing from the one selected by $\hat{\policy}$ in the same fiber yields a nonpositive advantage. Therefore, $\hat{\policy}$ is optimal and, since its value function is zero, the MDP is in normal form.
\end{proof}
\noindent 
Based on this result, we give the following definition.
\begin{definition}[\textbf{Normal set}]
    We define the normal set for a class of comparable MDPs as the subset $\normset \subset \maps(\U, \R)$ consisting of all the reward functions satisfying condition \eqref{eq:sufficient_condition_for_normality}:
    \begin{align}
        \normset \coloneqq \left\{\cost \in \maps(\U, \R) :~ \max_{\act \in \U_\state}\, \cost(\act) = 0, ~ \forall \state \in \X\right\}.
    \end{align}
\end{definition}
\noindent
Theorem~\ref{prop:nec_suff_condition_for_normality} motivates the term \qmarks{normal} set, as it is exactly the set of all reward functions in normal form. The following result characterizes the structure of this set and shows that the reward space $\maps(\U, \R)$ has a fiber bundle structure.
\begin{theorem}[\textbf{Characterization of the normal set}]\label{th:normal_set_characterization}
    Let $\normset$ be the normal set of a class $\mdparg{\cdot}{\disc \f}$ of model-coupled MDPs, equipped with the relative topology\footnote{The sets $\maps(\U, \R)$ and $\maps(\X, \R)$ are endowed with the topologies induced by the vector identification introduced in the previous section.}. Then:
    \begin{enumerate}[label=(\alph*), ref=\thetheorem-(\alph*)]
        \item\label{th:normal_set_sections} The normal set $\normset$ admits a decomposition as a union
        \begin{align}
            \normset = \!\!\bigcup_{\policy \in \Sect}\!\! \normset_\policy,
        \end{align}
        where $\normset_\policy$, defined as
        \begin{align}
            \normset_{\policy} \coloneqq \{\cost \in \maps(\U, \R) : \cost \leq 0, \costfb{\policy} = 0\},
        \end{align}
        is the set of normal rewards for which $\policy$ is optimal.
        \item\label{th:reward_space_as_principal_bundle} The reward space $\maps(\U, \R)$ is homeomorphic to the product $\normset \times \maps(\X, \R)$ and the map
        \begin{subequations}\label{eq:bundle_projection_map}
            \begin{align}
                \mathcal{P} : \maps(\U, \R) &\to \normset\\
                \cost &\mapsto \actAdv{\vvarg{\policy^*\!, \cost}}(\cost),
            \end{align}
        \end{subequations}
        is a continuous projection which makes it the total space of a principal bundle over $\normset$.
    \end{enumerate}
\end{theorem}
\begin{proof}
    $(a)$ The decomposition naturally follows from the definition of $\normset$, while the policy $\policy$ is optimal for any $\cost \in \normset_{\policy}$ since it satisfies $\cost_{\policy} = 0$ and has zero value function.\newline
    $(b)$ Since the action is faithful, each $\actAdvsym$-orbit is isomorphic to the group $\maps(\X, \R)$ and also intersects the normal set in just one point by Theorem~\ref{lem:uniqueness_of_normal_form}. Moreover, by Corollary~\ref{cor:ideal_transformation} and Theorem~\ref{lem:uniqueness_of_normal_form}, the map $\mathcal{P}$ defined above maps each $\actAdvsym$-orbit in its unique intersection point with $\normset$.
    Furthermore, it is surjective and idempotent, since for each $\cost \in \normset$, we have that
    \begin{align*}
        \mathcal{P}(\cost) = \actAdv{\vvarg{\policy^*\!,\cost}}(\cost) = \actAdv{0}(\cost) = \cost,
    \end{align*}
    and thus, $\mathcal{P}\rvert_{\normset} = \id{\normset}$. Therefore, we can identify $\maps(\U, \R)$ with the product $\normset\times \maps(\X, \R)$ through the mapping $\cost \mapsto \left( \mathcal{P}(\cost), \vvarg{\policy^*\!, \cost} \right)$ whose inverse is continuous and given by $(\cost_n, \delta) \mapsto \cost_n - \lintransf(\delta)$, where $\lintransf$ is the linear transform characterizing the $\actAdvsym$-action, see \eqref{eq:reward_transform}. It only remains to prove that $\mathcal{P}$ is continuous. Consider the quotient space $\maps(\U, \R)/_{\sim_{\actAdvsym}} \coloneqq \maps(\U, \R)/\maps(\X, \R)$ with quotient map~$q$, whose restriction $q\rvert_{\normset}$ is continuous and bijective by construction. The projection $\mathcal{P}$ can be expressed as $\mathcal{P} = (q\rvert_{\normset})^{-1} \circ q$, and it is continuous if and only if $(q\rvert_{\normset})^{-1}$ is continuous. %
    If $U \subseteq \normset$ is an arbitrary $\normset$-closed set it is also $\maps(\U, \R)$-closed, since $\normset$ is $\maps(\U, \R)$-closed. Furthermore,
    \begin{align*}
        q^{-1}(q(U)) = \{\cost : \cost = \cost_n - \lintransf(\delta), \cost_n \in U, \delta \in \maps(\X, \R)\},
    \end{align*}
    and it is therefore $\maps(\U, \R)$-closed, implying that $q(U)$ is $(\maps(\U, \R)/_{\sim_{\actAdvsym}})$-closed, being the preimage of a closed set through a quotient map. This further implies that $q$ is also a closed map and its restriction $q\rvert_{\normset}$ is a homeomorphism.
\end{proof}
\noindent The previous theorem has two main consequences:

\hspace{0.3cm}Part~(a), which holds for all comparable MDPs, implies that a policy $\policy$ is optimal for a given MDP if and only if the corresponding $\actAdvsym$-orbit intersects $\normset_\policy$. Note that the subsets $\normset_\policy$ are not disjoint. Their intersections correspond precisely to normal rewards that admit multiple optimal policies. In particular, if the $\actAdvsym$-orbits of two comparable MDPs intersect the normal set in $\normset_{\policy^*} \setminus \bigcup_{\policy \neq \policy^*} \normset_{\policy}$, then the two MDPs share the same unique optimal policy $\policy^*$. Hence, they can be regarded as equivalent with respect to the normalization process: replacing one with the other yields, after the normalization process, a greedy policy that is optimal for both. On the other hand, suppose the $\actAdvsym$-orbit of one MDP intersects $\normset_{\policy^*} \setminus \bigcup_{\policy \neq \policy^*} \normset_{\policy}$, while the $\actAdvsym$-orbit of the other intersects $\normset_{\policy^*} \cap \normset_{\policy}$, for some other $\policy \in \Sect$. Then they cannot be considered equivalent. In fact, normalizing the first still yields the unique optimal policy $\policy^*$, which remains optimal for the second (although the other optimal policy $\policy$ would be ignored). However, normalizing the second would yield two optimal policies, only one of which is optimal for the first. A tie-breaking rule could then select a policy that is suboptimal for the first. These considerations should be taken into account when dealing with uncertainty in the model used to normalize an unknown or partially known MDP (see next section). 

\hspace{0.3cm}Part~(b) implies that any reward function $\cost$ is uniquely determined by the pair consisting of the corresponding normal reward function $\cost_n$ and optimal value function $\vvarg{\policy^*\!, \cost}$, which leads to a natural identification:
\begin{align}\label{eq:identification_reward_as_tuple}
    \cost \leftrightarrow (\cost_n, \vvarg{\policy^*\!, \cost}).
\end{align}
Figure \ref{fig:orbits} shows the principal bundle geometry of a small dimensional MDP, chosen because it can be visualized explicitly.
\begin{figure}[!t]
    \centering
    \includegraphics[width = 0.8\textwidth]{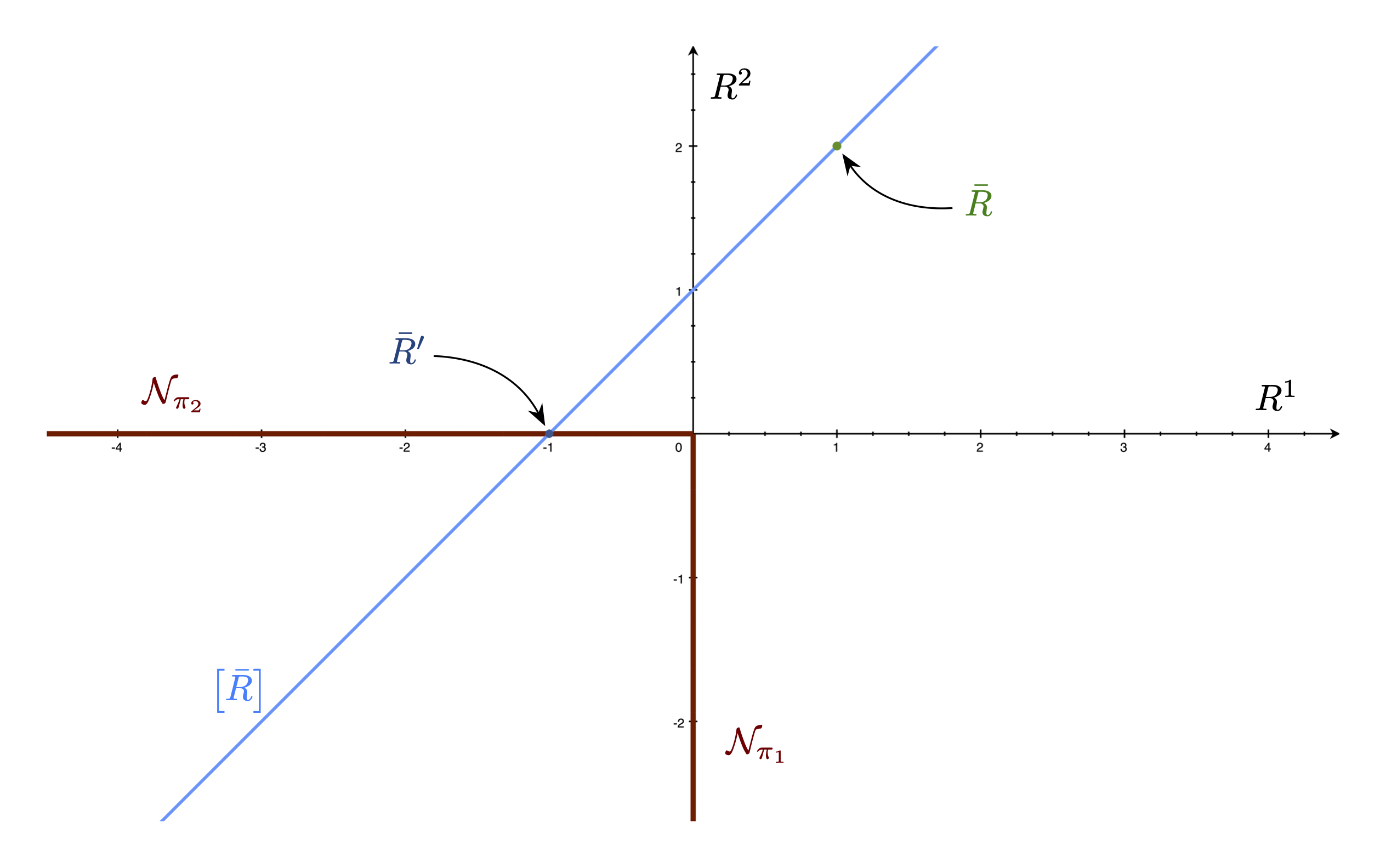}
    \caption{Representation of the space $\maps(\U, \R)$ of an MDP with just one state and two actions to choose from, for which the optimal policy is greedy regardless of the chosen reward and discount factor, with an optimal value function given by $\Varg{\policy^*\!, \cost}(\state) = \max\{\Cost^1, \Cost^2\}/(1 - \disc)$. The dark red set identifies a section of the normal set, in this case defined by $\normset = \normset_{\policy_1} \cup \normset_{\policy_2}$, where $\normset_{\policy_i} = \{\Cost \in \R^2 : \Cost^i = 0, \Cost^j \leq 0, j \neq i\}$, while the light blue line is the $\actAdvsym$-orbit corresponding to the reward $\bar{\Cost} = (1, 2)$. Notice that the orbit intersects the normal set in just one point, corresponding to its unique normalization $\bar{\Cost}^\prime = (-1, 0)$.}
    \label{fig:orbits}
\end{figure}

\subsubsection*{MDP Symmetry and Reward Transformations.}
We conclude this subsection by briefly characterizing reward transformations that, when the original MDP is acted upon by a group $G$, preserve the symmetry. Recalling Definition \ref{def:group_invariant_mdp} and \eqref{eq:group_invariant_mdp_vector_form}, we obtain the following result.
\begin{lemma}\label{lem:reward_transform_group_invariant_mdp}
    If $(\mdparg{\cost}{\disc \f}, G)$ is a $G$-invariant MDP, then:
    \begin{align}
        \bigl[\actAdv{\delta}(\cost)\bigr] \circ \reprarg{\mathrm{u}}{g} = \actAdv{\scalebox{0.6}{$\delta \circ \reprarg{\mathrm{x}}{g}$}}(\cost),
    \end{align}
    for all $\delta \in \maps(\X, \R)$, or, in vector representation:
    \begin{align}
        A_g^{\mathrm{u}}\left( \mathcal{L}^\Delta \Cost\right) = \mathcal{L}^{A_g^{\mathrm{x}}\Delta} \Cost.
    \end{align}
\end{lemma}
\begin{proof}
    Let us apply an arbitrary action $A_g^{\mathrm{u}}$ to both sides of the update rule \eqref{eq:reward_transform_representation}:
    \begin{align*}
        A_g^{\mathrm{u}} \left( \mathcal{L}^\Delta \Cost\right) &= A_g^{\mathrm{u}} \Cost + A_g^{\mathrm{u}}\left( \disc \F - \Proj \right) \Delta\\
        &\!\stkrl{a}{=} \Cost + \left( \disc \F A_g^{\mathrm{x}} - A_g^{\mathrm{u}}\Proj \right) \Delta\\
        &= \Cost + \left( \disc \F - A_g^{\mathrm{u}}\Proj\bigl(A_g^{\mathrm{x}}\bigr)^{\top} \right) A_g^{\mathrm{x}} \Delta\\
        &\!\stkrl{b}{=} \Cost + \left( \disc \F - \Proj \right) A_g^{\mathrm{x}} \Delta,
    \end{align*}
    where (a) follows from the invariance properties \eqref{eq:group_invariant_mdp_vector_form} together with the fact that each permutation matrix is an orthogonal matrix, implying $A_g^{\mathrm{u}} \F = \F A_g^{\mathrm{x}}$. Identity (b) follows from the structure of $A_g^{\mathrm{u}}$. On one hand it permutes the row blocks corresponding to the states. On the other hand it permutes the rows within each block, which correspond to the actions in the corresponding fiber. The latter permutation has no effect on the block-diagonal matrix $\Proj = \Projexpl$, since each block contains identical rows. Hence, the effect of $A_g^{\mathrm{u}}$ on this matrix is to permute the state blocks. This can be formalized by thinking of it as a diagonal matrix $D$, where each row corresponds to a block. In this case:
    \begin{align*}
        A_g^{\mathrm{x}} D (A_g^{\mathrm{x}})^\top
    \end{align*}
    remains diagonal, with its diagonal elements permuted according to the action $\reprarg{\mathrm{x}}{g}$. Finally, the group action $\reprarg{\mathrm{x}}{g}$ permutes states with the same number of actions, and then a diagonal block $\mathds{1}_{m_i}$ may only be replaced by a block $\mathds{1}_{m_j}$ with the same size $(m_i = m_j)$, ensuring that the block-diagonal structure is preserved under the group action.
\end{proof}
\noindent
Lemma~\ref{lem:reward_transform_group_invariant_mdp} finally implies the following characterization.
\begin{proposition}[\textbf{$\boldsymbol{G}$-invariance and reward transformations}]\label{prop:g_invariance}
    Let $(\mdparg{\cost}{\disc \f}, G)$ be a $G$-invariant MDP. A reward transformation $\actAdv{\delta}$ preserves $G$-invariance if and only if $\delta$ is $G$-invariant, i.e. $\delta \circ \reprarg{\mathrm{x}}{g} = \delta$ for all $g \in G$.
\end{proposition}
\begin{proof}
    $(\Rightarrow):$ If $\mathcal{L}^\Delta\Cost$ is invariant, $A_g^{\mathrm{u}}(\mathcal{L}^\Delta\Cost) = \mathcal{L}^\Delta\Cost$ for all $g \in G$. Then, by Lemma \ref{lem:reward_transform_group_invariant_mdp}, we have:
    \begin{align*}
        \underbrace{\left( \disc \F - \Proj \right)}_{= \B}\left( A_g^{\mathrm{x}} - I_n \right)\Delta = 0, \quad \forall g \in G,
    \end{align*}
    and, since $\B$ is injective:
    \begin{align*}
        (A_g^{\mathrm{x}} - I_n)\Delta = 0, \quad \forall g \in G.
    \end{align*}
    Therefore:
    \begin{align*}
        A_g^{\mathrm{x}}\Delta = \Delta ~\Longleftrightarrow~ \delta \circ \reprarg{\mathrm{x}}{g} = \delta.
    \end{align*}
    $(\Leftarrow):$ Follows directly from Lemma \ref{lem:reward_transform_group_invariant_mdp}.
\end{proof} %
\subsection{The Set of Admissible Transformations}
\label{subsec:section_1B}
So far, we have introduced and analyzed the group of advantage-preserving transformations and its action on $\maps(\U, \R)$. However, not every such transformation is useful for the purpose of (iteratively) normalizing the MDP. For instance, if the MDP is already in normal form, the only feasible transformation is the identity, corresponding to $\delta = 0$, by Theorem~\ref{lem:uniqueness_of_normal_form}. In particular, we are interested in those transformations which, when applied to a given reward function, shift the corresponding optimal value function closer to zero. By Proposition~\ref{lem:advantage_preservation}, this requirement on $\actAdv{\delta}$ can be expressed by asking for the corresponding (unique) group element $\delta$ to belong to the following set:
\begin{align}\label{eq:feasible_ball}
    \mathcal{B}_*(\cost) \coloneqq \left\{\delta : |\vvarg{\policy^*\!, \cost}(\state) - \delta(\state)| \leq |\vvarg{\policy^*\!, \cost}(\state)|, ~\forall \state \in \X\right\}.
\end{align}
\noindent
However, as we will see in the following section, it is convenient to further restrict our attention to nonpositive reward functions. Among other benefits, this restriction allows us to reformulate the normalization of the MDP as an optimal control problem using condition \eqref{eq:sufficient_condition_for_normality}. Unsurprisingly, it can always be assumed, without loss of generality, that the initial reward function is nonpositive, as justified by the following proposition, which we include for completeness.
\begin{proposition}\label{lem:wlog_reward_nonpositivity}
    Given any reward function $\cost$, the transformation
    \begin{align*}
        \cost(\act) \mapsto \cost(\act) - \max_{a \in \U} \cost(a) \leq 0
    \end{align*}
    is advantage-preserving.
\end{proposition}
\begin{proof}
    In vector notation, the transformation corresponds to 
    \begin{align*}
        \Delta = \frac{\max_j \Cost^j}{1 - \disc} \mathds{1}_n,
    \end{align*}
    since
    \begin{align*}
        \Cost^\prime &= \Cost + \frac{\max_j \Cost^j}{1 - \disc} (\disc \F - \Proj) \mathds{1}_n\\
        &= \Cost - \bigl(\max\nolimits_j \Cost^j\bigr) \mathds{1}_n
    \end{align*}
    where the last identity follows by the row-stochasticity of the $\F$ matrix.
\end{proof}
\noindent It is therefore reasonable to require admissible transformations to also preserve nonpositivity, namely to satisfy
\begin{align}
    \cost \leq 0 ~\implies~ \actAdv{\delta}(\cost) \leq 0,
\end{align}
which is equivalent to requiring $\delta$ to solve the following system of $m$ linear inequalities
\begin{align}\label{eq:system_inequalities_nonpositivity}
    \lintransf(\delta)(\act) \leq -\cost(\act), \quad \forall \act \in \U.
\end{align}
We denote by $\advset{\act}{\cost}$ the closed half-space defined by each of these inequalities and by $\advset{}{\cost}$ the set 
\begin{align}
    \advset{}{\cost} = \bigcap_{\act \in \U} \advset{\act}{\cost},
\end{align}
corresponding to the solutions of the system \eqref{eq:system_inequalities_nonpositivity}.
\begin{remark}\label{rem:geometric_actions}
    Each action can be identified with the boundary hyperplane of the corresponding half space $\advset{\act}{\cost}$ through the map 
    \begin{align*}
        \act \mapsto \partial\mathcal{C}_\act(\cost) = \{ \delta : \lintransf(\delta)(\act) = -\cost(\act) \}.
    \end{align*} 
    The coefficients of the defining linear-affine function $\cost(\act) + \lintransf(\cdot)(\act)$ for $\partial\mathcal{C}_\act(\cost)$ determine the entries of the action vector corresponding to $\act$, as introduced in \cite{mustafin2025mdpgeometrynormalizationreward}. Consequently, the actions of the MDP define an arrangement of hyperplanes in the ``value space'' $\maps(\X, \R)$. Moreover, for any $\policy$ it holds
    \begin{align}
        \{\vv\} = \bigcap_{\state \in \X} \partial\mathcal{C}_{\policy(\state)}(\cost).
    \end{align}
\end{remark}
\noindent The discussion above motivates the following definition.
\begin{definition}[\textbf{The admissible set}]
    For any nonpositive reward function $\cost$, we define the admissible set of transformations for the normalization problem as
    \begin{align}\label{eq:admissible_set_def}
        \feasset{\cost} &\coloneqq \advset{}{\cost} \cap \mathcal{B}_*(\cost).
    \end{align}
\end{definition}
\noindent We summarize the key properties of the admissible set in the following proposition.
\begin{proposition}[\textbf{Properties of the admissible set}]\label{prop:properties_feasible_set}
    For any nonpositive reward function $\cost$, the admissible set $\feasset{\cost}$ is nonempty, compact, and convex.
\end{proposition}
\begin{proof}
    The admissible set is nonempty since both the zero function and, by Corollary~\ref{cor:ideal_transformation}, the optimal value function belong to this set. The fact that it is compact and convex follows immediately from the fact that it is the intersection between two closed convex sets, one of which, $\mathcal{B}_*(\cost)$, is compact. 
\end{proof}
\noindent
Observe that the nonpositivity of the reward function implies the nonpositivity of each value function, see \eqref{eq:value_funct_def}, including the optimal one. Consequently, the set $\mathcal{B}_*(\cost)$ is a compact subset of the nonpositive orthant. Together with the following proposition, this provides a clear geometric interpretation of each function $\delta$ belonging to this set.
\begin{proposition}[\textbf{Admissible inputs as overestimates of the optimal value function}]\label{prop:feas_set_as_overestimates}
    For any nonpositive reward function $\cost$, every function in $\advset{}{\cost}$ is an overestimate of the corresponding optimal value function. Formally:
    \begin{align*}
        \advset{}{\cost} \subseteq \mathcal{C}_*(\cost),
    \end{align*}
    where
    \begin{align}\label{eq:overestimates_set}
        \mathcal{C}_*(\cost) \coloneqq \{\delta \in \maps(\X, \R) : \delta \geq \vvarg{\policy^*\!, \cost}\}.
    \end{align}
\end{proposition}
\begin{proof}
    If $\delta \in \advset{}{\cost}$ then
    \begin{align}\label{eq:inequalities_feasible_set}
        \costfb{\policy}(\state) + (\disc \text{$\ffb{\policy}(\state, \state)$} - 1) \delta(\state) + \sum_{y \neq \state} \disc \ffb{\policy}(y, \state) \delta(y) \leq 0,
    \end{align}
    for every state $\state \in \X$ and any fixed policy $\policy \in \Sect$. Since $\vvarg{\policy\!, \cost}$ satisfies \eqref{eq:inequalities_feasible_set} with equality (by the Bellman equation \eqref{eq:bellman_equation_v}), we can rewrite \eqref{eq:inequalities_feasible_set} as
    \begin{align}\label{eq:error_inequalities}
        \sum_{y \in \X} \disc \ffb{\policy}(y, \state) \tilde{\delta}_{\policy, \cost}(y) \leq \tilde{\delta}_{\policy, \cost}(\state),
    \end{align}
    where we have introduced the shorthand notation $\tilde{\delta}_{\policy, \cost} \coloneqq \delta - \vvarg{\policy\!, \cost}$. Since the left-hand side of \eqref{eq:error_inequalities} is a convex combination scaled by $\disc$, we have
    \begin{align*}
        \disc \min_{z \in \X} \tilde{\delta}_{\policy, \cost}(z) \leq \sum_{y \in \X} \disc \ffb{\policy}(y, \state) \tilde{\delta}_{\policy, \cost}(y) \leq \tilde{\delta}_{\policy, \cost}(\state).
    \end{align*}
    Now, since this holds for all $\state \in \X$, we can conclude that
    \begin{align*}
        \disc \min_{z \in \X} \tilde{\delta}_{\policy, \cost}(z) \leq \min_{x \in \X} \tilde{\delta}_{\policy, \cost}(x),
    \end{align*}
    or
    \begin{align*}
        (1 - \disc)\min_{z \in \X} \tilde{\delta}_{\policy, \cost}(z) \geq 0,
    \end{align*}
    which, since $\disc < 1$, implies that $\tilde{\delta}_{\policy, \cost}(\state) = \delta(\state) - \vvarg{\policy\!, \cost}(\state) \geq 0, \forall \state$. Since the policy $\policy$ is arbitrary, this holds for the optimal one as well, and this concludes the proof.
\end{proof}
\noindent
We can therefore interpret the admissible set as the collection of all nonpositive overestimates of the optimal value function $\vvarg{\policy^*\!, \cost}$, whose corresponding transformations preserve the nonpositivity of the reward function. This finally yields the following characterization.
\begin{proposition}[\textbf{Characterization of the admissible set via linear inequalities}]\label{prop:alternative_definition_feasible_set}
    The admissible set of any nonpositive reward function $\cost$ can be expressed in terms of linear inequalities as follows:
    \begin{align}\label{eq:alternative_definition_feasible_set}
        \feasset{\cost} = \left\{\delta \in \maps(\X, \R) : \delta \leq 0, ~ \actAdv{\delta}(\cost) \leq 0\right\},
    \end{align}
    where $\actAdv{\delta} : \maps(\U, \R) \to \maps(\U, \R)$ is defined as in \eqref{eq:reward_transform}.
\end{proposition}
\begin{proof}
    If $\mathcal{C}_*(\cost)$ denotes the set of overestimates of the optimal value function as in \eqref{eq:overestimates_set}, and $\orthant$ denotes the nonpositive orthant of $\maps(\X, \R)$, then we have
    \begin{align*}
        \mathcal{C}_*(\cost) \cap \mathcal{B}_*(\cost) = \mathcal{C}_*(\cost) \cap \orthant,
    \end{align*}
    since $\mathcal{B}_*(\cost) \subseteq \orthant$ trivially implies $\mathcal{C}_*(\cost) \cap \mathcal{B}_*(\cost) \subseteq \mathcal{C}_*(\cost) \cap \orthant$, while:
    \begin{align*}
        \delta \in \mathcal{C}_*(\cost) \cap \orthant &\Longleftrightarrow 0 \geq \delta(\state) \geq \vvarg{\policy^*\!, \cost}(\state),~ \forall \state \in \X\\
        &\Longleftrightarrow -\vvarg{\policy^*\!, \cost}(\state) \geq \delta(\state) - \vvarg{\policy^*\!, \cost}(\state) \geq 0,~ \forall \state \in \X\\
        &\Longrightarrow \delta \in \mathcal{B}_*(\cost),
    \end{align*}
    or, equivalently, $\mathcal{C}_*(\cost) \cap \orthant \subseteq \mathcal{C}_*(\cost) \cap \mathcal{B}_*(\cost)$. This implies, by Proposition~\ref{prop:feas_set_as_overestimates}, that
    \begin{align*}
        \advset{}{\cost} \cap \mathcal{B}_*(\cost) = \advset{}{\cost} \cap \orthant.
    \end{align*}
\end{proof}
\begin{center}
    \begin{figure}[t]
        \includegraphics[width = 0.8\textwidth]{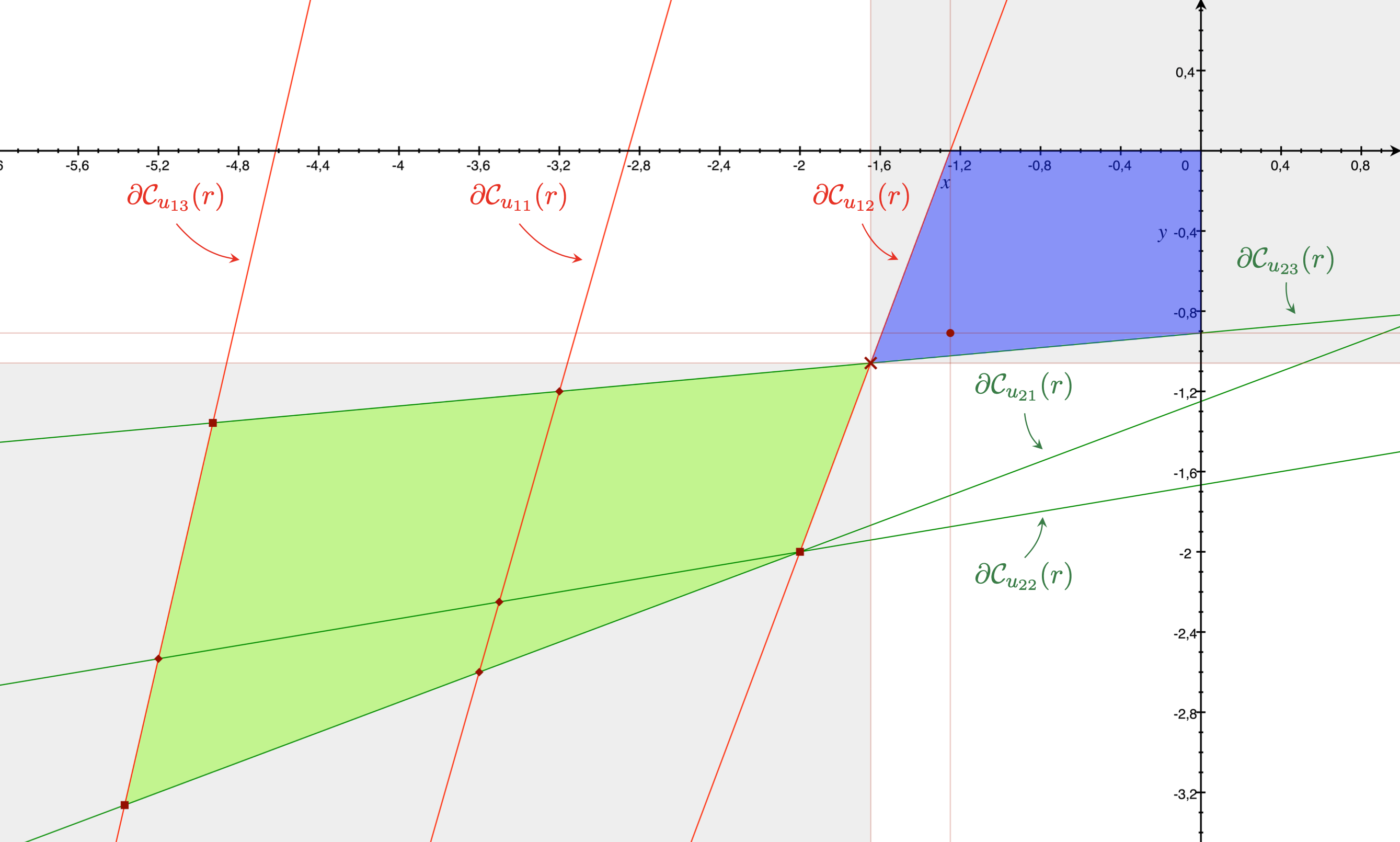}
        \caption{Visualization of the value space $\maps(\X, \R)$ for an MDP with two states $x_1, x_2$ and an action space $\{u_{ij} : i \in \{1, 2\}, j \in \{1, 2, 3\}\}$, where $u_{ij} \in \U_{x_i}$. Each intersection point $\cred{\partial \mathcal{C}_{u}(\cost)} \cap \cgreen{\partial \mathcal{C}_{v}(\cost)}$, with $u \in \U_{x_1}$ and $v \in \U_{x_2}$, is the value function corresponding to the policy $(x_1, x_2) \mapsto (u, v)$. The optimal value function is marked with a red cross, whereas the admissible set is the blue polygon. The green polygon is the convex hull of the set of deterministic value functions, which includes value functions corresponding to stochastic policies as well, and the light-grey cones represent the set of underestimates (bottom-left) and overestimates (top-right) of the optimal value function.}
      \label{fig:feasible_set}
    \end{figure}
\end{center} \section{Control-Theoretic Reformulation}
\label{sec:section_2}
In this section, we provide a control-theoretic reformulation of the normalization problem. Specifically, in Subsection~\ref{subsec:section_2A} we formulate this control problem, while in Subsection~\ref{subsec:section_2B} we characterize control laws achieving asymptotic normalization, and provide some examples thereof. Finally, in Subsection~\ref{subsec:section_2C}, we reformulate the normalization problem as an optimal control problem and provide sufficient conditions on its feasibility and on the existence of solutions in the infinite-horizon case.
\subsection{Normalization as an Output Stabilization Problem}
\label{subsec:section_2A}
In the previous section, we introduced the reward-dependent set of transformations to be used in the design of a normalizing algorithm for a given MDP. Now, we reformulate this task within a system-theoretic framework, with the idea of restating it as a control problem. To this end, consider the vector representation \eqref{eq:reward_transform_representation} of a general advantage-preserving reward transformation and let us rewrite it as an iterative law
\begin{align}\label{eq:reward_transform_system}
    \Cost_{t+1} = \Cost_t + \B \Delta_t
\end{align}
where
\begin{align}\label{eq:input_matrix}
    \B \coloneqq \disc \F - \Proj,
\end{align}
and $t \in \N$. In doing so, we are considering the reward vector $\Cost_t$ as the state of a linear time-invariant system and $\Delta_t$ as a control input. Hence, our problem boils down to searching for a feedback control law $\Cost \mapsto \Delta_{\sml \text{FB}}(\Cost)$ able to steer the reward $\Cost$, initialized in the nonpositive orthant, towards the normal set. The normalization problem can be interpreted as an output stabilization problem, with output map given by
    \begin{align}
    \Cost \mapsto \underbrace{(I_n - \disc \Ffb{\policy^*})^{-1} \Pi^* \Cost}_{= \Varg{\policy^*\!, \cost}} \label{eq:ideal_output_map}
\end{align}
where $\policy^*$ is any optimal policy for the MDP. However, this output map is impractical, since computing it explicitly is essentially equivalent to solving the very problem the normalization process is meant to address. Therefore, we leverage Theorem~\ref{prop:nec_suff_condition_for_normality} and replace \eqref{eq:ideal_output_map} with a measurable nonlinear output map $\outmap : \R^m \to \R^n$ with components defined by
\begin{align}\label{eq:output_regulation_implicit}
    \outmap^i(\Cost) = \max_{j \in \U_i} \Cost^j, \quad i \in \{1, \hdots, n\}.
\end{align}
Consequently, we consider the dynamical system
\begin{subequations}\label{eq:system_dynamics}
    \begin{align}
        \Cost_{t+1} &= \Cost_t + \B \Delta_t,\\
        y_t &= \outmap(\Cost_t),
    \end{align}
\end{subequations}
and seek a sequence of admissible control inputs $\{\Delta_t \in \feasset{\Cost_t}\}_{t \in \N}$. We will refer to such a sequence as a \emph{normalizing control law}. 

\subsection{Normalizing Control Laws}
\label{subsec:section_2B}
Before providing examples of normalizing control laws, we first present a characterization thereof.
\begin{proposition}[\textbf{Characterization of normalizing controls}]\label{cor:necessary_sufficien_condition_for_normalizing_control}
    A control law $t \mapsto \Delta_t$ for the dynamics \eqref{eq:reward_transform_system} is normalizing if and only if it satisfies
    \begin{align*}
        \sum_{t = 0}^\infty \Delta_t = \Varg{\policy^*\!, \cost_0}.
    \end{align*}
\end{proposition}
\begin{proof}
    By Theorem~\ref{th:reward_space_as_principal_bundle}, every reward vector $\Cost$ can be uniquely identified with a tuple $(\bar{\Cost}, \Varg{\policy^*\!, \cost})$, where $\bar{\Cost}$ is its normal representative and $\Varg{\policy^*\!, \cost}$ is its optimal value function. In particular, it can be expressed as
    \begin{align*}
        R = \bar{R} - B \Varg{\policy^*\!, \cost}.
    \end{align*}
    We can exploit this fact to easily put the dynamics in error coordinates with respect to the unique normal representative. In particular, we have:
    \begin{align*}
        \Cost_{t+1} &= \Cost_t + B \Delta_t\\
        &= \Cost_0 + B \sum_{\tau = 0}^t \Delta_\tau\\
        &= \bar{\Cost} - B \Varg{\policy^*\!, \cost_0} + B\sum_{\tau = 0}^t \Delta_\tau,
    \end{align*}
    and thus, by introducing $\tilde{\Cost} \coloneqq \Cost - \bar{\Cost}$, we finally obtain:
    \begin{align}\label{eq:dynamics_error_coordinates}
        \tilde{\Cost}_{t + 1} = - B \left(\Varg{\policy^*\!, \cost_0} - \sum_{\tau = 0}^t \Delta_\tau\right).
    \end{align}
    The statement then follows from the injectivity of $B$.
\end{proof}
\noindent
In other words, in order to normalize an MDP with reward function $\cost$, a control law $\{\Delta_t\}_{t \in \N}$ must constitute a partition of the corresponding optimal value function $\vvarg{\policy^*\!, \cost}$. We now present three examples of normalizing control laws. Two of them were originally introduced in~\cite{mustafin2025mdpgeometrynormalizationreward} as normalization algorithms: here, we recast them as control laws within our system-theoretic framework and establish additional structural properties.
\newline
\noindent\textbf{\emph{I) Ideal Normalizing Control.}}
\newline
The first law we present is impractical, yet it serves as the prototype for all normalizing control laws.
\begin{definition}\label{def:ideal_output_feedback}
    We define the ideal output-feedback control law by setting
    \begin{align}
        \Delta_t = \Varg{\policy^*\!, \cost_t}.
    \end{align}
\end{definition}
\noindent This control law is characterized by the following result, which is an immediate consequence of Corollary~\ref{cor:ideal_transformation} and Proposition~\ref{cor:necessary_sufficien_condition_for_normalizing_control}.
\begin{proposition}
    The ideal output-feedback control law is admissible and normalizing. In particular, it solves the problem in one step.
\end{proposition}
\noindent It follows that the corresponding input sequence is given by:
\begin{align}\label{eq:ideal_input_sequence}
    \Delta_t = 
    \begin{cases}
        (I_n - \disc \Ffb{\policy^*})^{-1} \Pi^* \Cost_0, & t = 0\\
        0, & t > 0
    \end{cases}
\end{align}
\noindent\textbf{\emph{II) Full-Output Feedback.}}
\newline
The second law we consider is the one used in~\cite[Algorithm 2]{mustafin2025mdpgeometrynormalizationreward}, where it was originally applied to an approximate-model setting.
\begin{definition}\label{def:MFRBS}
    The full-output feedback reward-balancing method is the control law obtained by setting
    \begin{align}
        \Delta_t = y_t.
    \end{align}
\end{definition}
\noindent We then provide the following result, which not only establishes its normalizing properties but also offers a clear interpretation linking this method to standard dynamic programming.
\begin{proposition}
    The full-output feedback reward-balancing control law is admissible, normalizing, and equivalent to value iteration\footnote{For details about value iteration see, e.g., \cite[Chapter 6.3]{puterman2014markov}.} with initial estimate $\hat{V}_0 = \outmap(\Cost_0)$.
\end{proposition}
\begin{proof}
    The control law is feasible by Proposition~\ref{prop:alternative_definition_feasible_set}, since if $\Cost_t \leq 0$ then clearly, for every $i \in \{1, \hdots, n\}$, we have $\Delta_t^i = \max_{l \in \U_i} \Cost^l_t \leq 0$, and, for every $j \in \U_i$
    \begin{align}
        \Cost^j_{t + 1} &= \Cost_t^j - \Delta^i_t + \disc \sum_{k = 1}^n \F_k^j \Delta_t^k\\
        &= \Cost_t^j - \max_{l \in \U_i} \Cost^l_t + \disc \sum_{k = 1}^n \F_k^j \max_{l \in \U_k} \Cost_t^l\\
        &\leq 0.
    \end{align}
    Furthermore:
    \begin{align*}
        \smash{\underbrace{\max_{j \in \U_i} \Cost^j_{t + 1}}_{= \Delta_{t + 1}^i}} &= \max_{j \in \U_i} \left\{\Cost^j_t - \Delta^i_t + \disc \sum_{k = 1}^n \F_k^j \Delta_t^k\right\} \\
        &= \max_{j \in \U_i} \left\{\Cost_0^j - \sum_{\tau = 0}^t \Delta_\tau^i + \disc \sum_{k = 1}^n \F_k^j \sum_{\tau = 0}^t \Delta_\tau^k\right\}\\
        &= \max_{j \in \U_i} \left\{\Cost_0^j + \disc \sum_{k = 1}^n \F_k^j \sum_{\tau = 0}^t \Delta_\tau^k\right\} - \sum_{\tau = 0}^t \Delta_\tau^i,
    \end{align*}
    which yields
    \begin{align}\label{eq:value_iteration}
        \sum_{\tau = 0}^{t+1} \Delta_\tau^i &= \max_{j \in \U_i} \left\{\Cost_0^j + \disc \sum_{k = 1}^n \F_k^j \sum_{\tau = 0}^t \Delta_\tau^k\right\}.
    \end{align}
    By setting $\hat{V}_t^i \coloneqq \sum_{\tau = 0}^{t} \Delta_\tau^i$, the update \eqref{eq:value_iteration} can be compactly written as
    \begin{align}
        \hat{V}_{t+1} = \bellman\bigl(\hat{V}_t\bigr),
    \end{align}
    where $\bellman : \R^n \to \R^n$ is the well known optimal Bellman operator, which is a $\disc$-contraction with respect to $\infty$-norm (see, e.g., \cite[Proposition 6.2.4]{puterman2014markov}). Consequently, \eqref{eq:value_iteration} is precisely the value iteration update on $\hat{V}_t$, which is a fixed-point iteration of the $\disc$-contraction $\bellman$. It follows that
    \begin{align}\label{eq:geometric_bound_on_output_feedback}
        \infnorm{\Varg{\policy^*\!, \cost_0} - \sum_{\tau = 0}^t \Delta_\tau} &\leq \disc^t \infnorm{\Varg{\policy^*\!, \cost_0} - \Delta_0}\notag\\
        &= \disc^t \max_{i} \left| \Varg{\policy^*\!, \cost_0}^i - \max_{j \in \U_i} \Cost_0^j\right|,
    \end{align}
    and thus, by Proposition~\ref{cor:necessary_sufficien_condition_for_normalizing_control}, the control law is normalizing.
\end{proof}
\noindent\textbf{\emph{III) Safe Reward-Balancing}.}
\newline
The third law we present is the one used in \cite[Algorithm 1]{mustafin2025mdpgeometrynormalizationreward}, referred to as \emph{safe reward-balancing} algorithm, a name we also retain here.
\begin{definition}
    The safe reward-balancing (RB-S) method is the nonlinear state-feedback control law obtained by setting
    \begin{align}\label{eq:safe_reward_balancing}
        \Delta^i_t = \max_{j \in \U_i} \frac{\Cost^j}{1 - \disc F_i^j}, \quad i \in \{1, \hdots, n\}.
    \end{align}
\end{definition}
\noindent The properties of this control law were already established in the analysis in the aforementioned paper. We nevertheless provide them here for completeness.
\begin{proposition}
    The safe reward-balancing control law is admissible and normalizing.
\end{proposition}
\begin{proof}
    The control law is admissible by Proposition~\ref{prop:alternative_definition_feasible_set} since clearly if $\Cost_t \leq 0$, then $\Delta_t \leq 0$ and, by \cite[Proposition C.1]{mustafin2025mdpgeometrynormalizationreward}, $\Cost_{t+1} \leq 0$. Moreover, as shown in the proof of \cite[Theorem 4.4]{mustafin2025mdpgeometrynormalizationreward}, the RB-S control law satisfies
    \begin{align}\label{eq:rbs_reconstructs_v_star}
        \infnorm{\Varg{\policy^*\!, \cost_0} - \sum_{\tau = 0}^{t-1} \Delta_\tau} \leq \max_i \left\{ - y_0^i \right\}\frac{\disc^t}{1 - \disc},
    \end{align}
    which in turn implies by Proposition~\ref{cor:necessary_sufficien_condition_for_normalizing_control} that it is normalizing.
\end{proof}
\noindent Figure \ref{fig:feasible_set} provides a geometric intuition of this input choice (the red dot within the admissible set) and it illustrates that, for each nonpositive reward, it corresponds to the best overestimate of the optimal value function whose projections onto the elements of the canonical basis of $\R^n$ remain feasible. This geometric intuition is formalized by the following proposition.
\begin{proposition}\label{prop:rbs_components_property}
    For every nonpositive $\Cost_t$, the RB-S control law satisfies
    \begin{align}\label{eq:constraint_rbs}
        \Delta^i_t = \min \{\alpha \in \R : \Cost_t + \alpha B \hat{e}_i \leq 0\}, \quad \forall i \in \{1, \hdots, n\}.
    \end{align}
\end{proposition}
\begin{proof}
    Consider the $i$-th state and the $j$-th constraint of \eqref{eq:constraint_rbs}. Then:
    \begin{enumerate}[label = $\rhd$]
        \item if $j \in \U_i$ we have
        \begin{align}\label{eq:lower_bounds_alpha}
            \Cost^j + \alpha (\disc \F_i^j - 1) \leq 0 ~\implies~ \alpha \geq \frac{\Cost^j}{1 - \disc \F_i^j}.
        \end{align}
        \item if $j \notin \U_i$ we have
        \begin{align}\label{eq:upper_bounds_alpha}
            \Cost^j + \alpha \disc \F_i^j \leq 0 ~\implies~ \alpha \disc \F_i^j \leq \underbrace{- \Cost^j}_{\geq 0}.
        \end{align}
    \end{enumerate}
    Now, every lower bound in \eqref{eq:lower_bounds_alpha} is nonpositive and therefore also satisfies every inequality \eqref{eq:upper_bounds_alpha}. Hence, the minimum feasible value for $\alpha$ is the maximum among these lower bounds, which is exactly the $i$-th entry of the RB-S control law.
\end{proof}
\begin{figure}[!t]
    \centering
    \includegraphics[width = 0.8\textwidth]{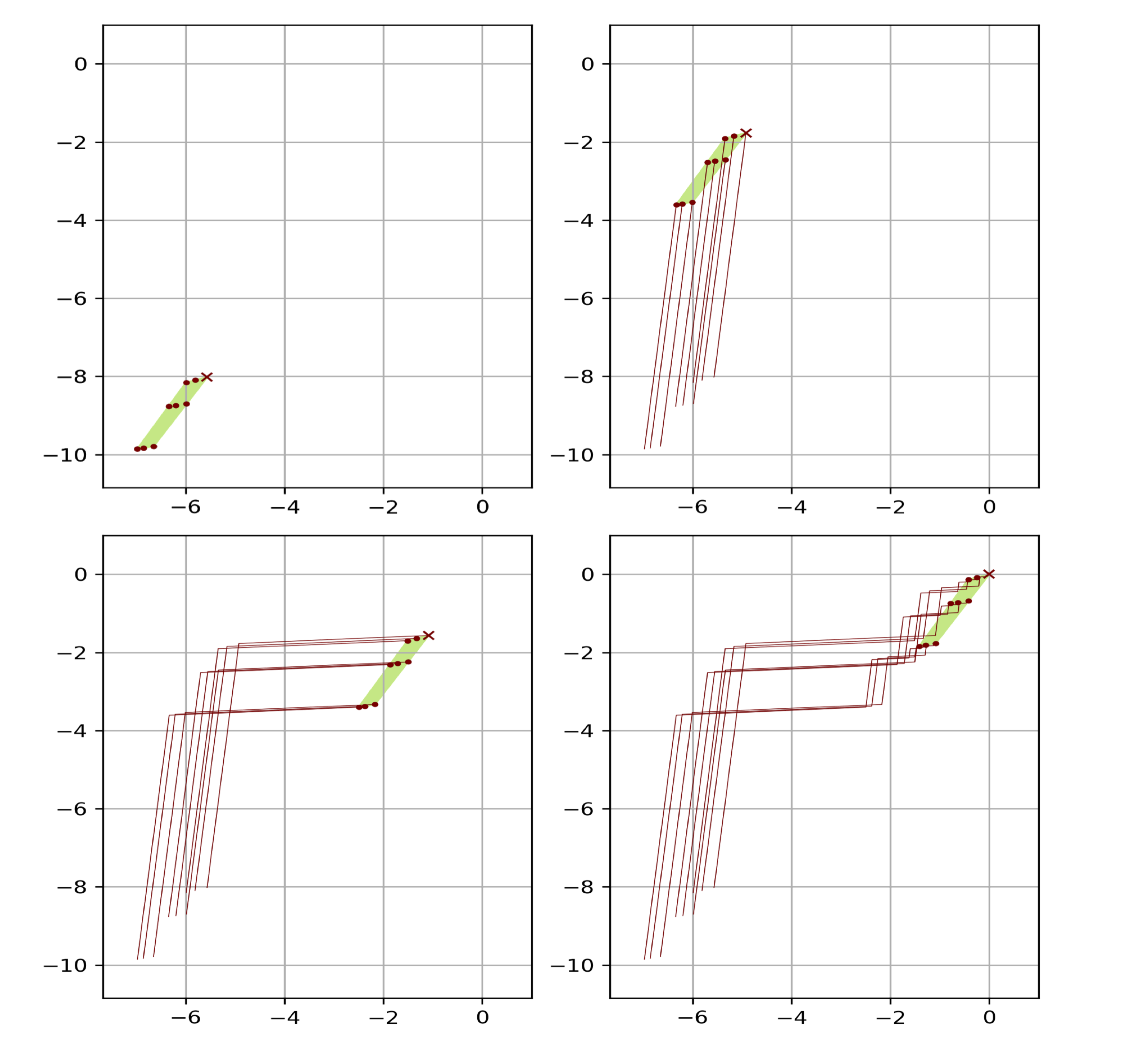}
    \caption{Example of the evolution of the set of value functions (green region, including stochastic policies, with the optimal one marked with a cross), for an MDP with two states and six actions in total (three available in each state) when the reward is updated using the RB-S control law computed using the exact model. Each suboptimal value function corresponding to a stationary deterministic policy is marked with a red dot, whereas the optimal one is marked with a cross. Depicted are the original set (top-left) and the set after one (top-right), two (bottom-left) and thirty iterations (bottom-right).}
    \label{fig:value_functions_evolution_exact_model}
\end{figure} %
\subsection{Normalization as an Optimal Control Problem}
\label{subsec:section_2C}
In this subsection, we recast the normalization problem as an optimal control problem, which can be used to design normalizing control laws in either a finite-horizon setting (where normalization can be achieved approximately, or exactly via model predictive control) or in an infinite-horizon setting. This formulation also allows for incorporating a broader variety of constraints and possibly additional penalty terms in the cost. 
Let us consider the following Optimal Control Problem (OCP)
\begin{subequations}\label{eq:ideal_optimal_control_problem}
    \begin{align}
        &\min_{\sml \{(\Cost_t, \Delta_t)\}} -\sum_{t = 1}^T \sum_{i = 1}^n \max_{j \in \U_i} \Cost_t^j \\\notag
        &\quad\text{subj.to:}\\
        &\qquad \Cost_0 = \bar{\Cost} \leq 0\\
        &\qquad \Cost_{t+1} = \Cost_t + B \Delta_t,\\
        &\qquad P \Cost_t + Q \Delta_t \leq 0,\label{eq:general_feas_constraint}\\
        &\quad\text{for all } t \in \{0, \hdots, T-1\}\notag
    \end{align}
\end{subequations}
where $T \in \N_+ \cup \{\infty\}$ and $P \in \R^{d\times m}, Q \in \R^{d\times n}$ satisfy
\begin{align}\label{eq:feasibility_of_general_constraint}
    \emptyset \neq \mathcal{F}(\Cost) \coloneqq \{\Delta : P \Cost + Q \Delta \leq 0\} \subseteq \feasset{\Cost},
\end{align}
for all $\Cost \leq 0$. That is, we allow for restrictions to a subset of the admissible set. In particular, if we choose:
\begin{align}\label{eq:ideal_constraint_choice}
    P = 
    \begin{bmatrix}
        0\\
        I_m
    \end{bmatrix}, 
    \quad 
    Q = 
    \begin{bmatrix}
        I_n\\
        B
    \end{bmatrix},
\end{align}
with $d = m + n$, we retrieve the admissible set by Proposition~\ref{prop:alternative_definition_feasible_set}. In this case, the optimal solution is trivial as shown by the following result.
\begin{proposition}\label{prop:ideal_feedback_control_optimality}
    If $P$ and $Q$ are chosen as in \eqref{eq:ideal_constraint_choice}, then the input sequence \eqref{eq:ideal_input_sequence} is the unique optimal solution of Problem \eqref{eq:ideal_optimal_control_problem}, for every $T \geq 1$. 
\end{proposition}
\begin{proof}
    The input sequence \eqref{eq:ideal_input_sequence} normalizes the MDP in a single step, and the corresponding \emph{stage cost}
    \begin{align}\label{eq:stage_cost}
        c_t \coloneqq -\sum_{i = 1}^n \max_{j \in \U_i} \Cost_t^j
    \end{align}
    is thus zero for all $t \geq 1$. Since the reward is constrained to remain nonpositive, the cost is lower-bounded by zero and this input sequence is therefore optimal. Any other optimal solution would also need to normalize the MDP in a single step as well (by Theorem~\ref{prop:nec_suff_condition_for_normality}). However, Proposition~\ref{cor:necessary_sufficien_condition_for_normalizing_control} shows that there is no $\Delta_0$ other than $\Varg{\policy^*\!, \cost_0}$ capable of achieving this.
\end{proof}
\noindent Although known in closed form, the previous solution is impractical, as already discussed. Moreover, in the presence of model uncertainties, more conservative solutions become particularly valuable. For instance, one could make the solution more robust by restricting the feasible set or by introducing additional penalty terms in the cost. Furthermore, if the MDP exhibits symmetries, it may be desirable to enforce the input to satisfy Proposition~\ref{prop:g_invariance}, which can be expressed as a finite set of constraints of the type $\pm (A_g^{\mathrm{x}} - I_n) \Delta \leq 0$, compatible with the structure of \eqref{eq:general_feas_constraint}. For general choices of $P$ and $Q$, the problem admits no solution in closed form, and exact normalization of the MDP can be achieved only asymptotically in the infinite-horizon case. In this setting, we provide two sufficient conditions on $P$ and $Q$ that guarantee feasibility of the problem and existence of optimal solutions.
\begin{theorem}[\textbf{Feasibility of the infinite-horizon OCP}]\label{th:sufficient_conditions_feasibility}
    Problem \eqref{eq:ideal_optimal_control_problem} with $T = \infty$ is feasible if, for every $\Cost \leq 0$, at least one of the following conditions holds:
    \begin{enumerate}[label= (\alph*), ref=\thetheorem-(\alph*)]
        \item\label{prop:sufficient_condition_feasibility_1} There exists $\alpha \in [0, 1)$ such that:
        \begin{align}
            \min_{\scalebox{0.6}{$\Delta : P \Cost + Q \Delta \leq 0$}} \infnorm{\Varg{\policy^*\!, \cost} - \Delta } \leq \alpha \infnorm{\Varg{\policy^*\!, \cost}}.
        \end{align}
        \item\label{prop:sufficient_condition_feasibility_2} The output map \eqref{eq:output_regulation_implicit} satisfies
        \begin{align}\label{eq:feasibility_of_output_feedback}
            P \Cost + Q \outmap(\Cost) \leq 0.
        \end{align}
    \end{enumerate}
    Moreover, under either condition, an optimal solution exists and has finite cost.
\end{theorem}
\begin{proof}
    (a) If we choose a feasible control law satisfying
    \begin{align*}
        \Delta_t \in \argmin_{\scalebox{0.6}{$\Delta : P \Cost_t + Q \Delta \leq 0$}} \infnorm{\Varg{\policy^*\!, \cost_t} - \Delta},
    \end{align*}
    we have
    \begin{align*}
        \infnorm{\Varg{\policy^*\!, \cost_0} - \sum_{\tau = 0}^t \Delta_\tau} &= \infnorm{\Varg{\policy^*\!, \cost_t} - \Delta_t}\\
        &\leq \alpha \infnorm{\Varg{\policy^*\!, \cost_t}}\\
        &= \alpha \infnorm{\Varg{\policy^*\!, \cost_{t-1}} - \Delta_{t-1}}\\
        &\leq \alpha^{t + 1} \infnorm{\Varg{\policy^*\!, \cost_0}}
    \end{align*}
    where we iteratively used the relation
    \begin{align*}
        \Varg{\policy^*\!, \cost_t} = \Varg{\policy^*\!, \cost_{t-1}} - \Delta_{t-1}.
    \end{align*}
    It follows that the stage cost \eqref{eq:stage_cost} is geometrically bounded, since
    \begin{align*}
        -\max_{j \in \U_i} \Cost_{t+1}^j &= \Delta_t^i - \max_{j \in \U_i}\left\{ \Cost_t^j + \sum_{k = 1}^n \disc \F_k^j \Delta_t^k \right\}\\
        &\stkrl{1}{=} \sum_{\tau = 0}^t \Delta_\tau^i - \max_{j \in \U_i}\left\{ \Cost_0^j + \sum_{k = 1}^n \disc \F_k^j \sum_{\tau = 0}^t \Delta_\tau^k \right\}\\
        &= \sum_{\tau = 0}^t \Delta_\tau^i - \left[\mathcal{T}\left(\sum_{\tau = 0}^t \Delta_\tau\right)\right]^i - \underbrace{\left[\Varg{\policy^*\!, \cost_0} - \mathcal{T}\left( \Varg{\policy^*\!, \cost_0} \right) \right]^i}_{= 0}\\
        &\leq \infnorm{\Varg{\policy^*\!, \cost_0} - \sum_{\tau = 0}^t \Delta_\tau} \hspace{-0.25cm}+ \infnorm{\mathcal{T}\left( \Varg{\policy^*\!, \cost_0} \right) - \mathcal{T}\left(\sum_{\tau = 0}^t \Delta_\tau\right)}\\
        &\leq (1 + \disc) \infnorm{\Varg{\policy^*\!, \cost_0} - \sum_{\tau = 0}^t \Delta_\tau}\\
        &\leq (1 + \disc) \alpha^{t + 1} \infnorm{\Varg{\policy^*\!, \cost_0}},
    \end{align*}
    where in (1) we used 
    \begin{align*}
        \Cost_t^j = \Cost_0^j - \sum_{\tau = 0}^{t-1} \left(\Delta_\tau^i - \sum_{k = 1}^n \disc \F_k^j \Delta_\tau^k\right),
    \end{align*}
    and thus the summation $\sum_{t = 1}^\infty c_t$ converges.\newline
    (b) If the output map $\outmap$, as introduced in \eqref{eq:output_regulation_implicit}, satisfies \eqref{eq:feasibility_of_output_feedback}, then the feedback law $\Delta_t = y_t$ is feasible for the optimal control problem. It follows from~\eqref{eq:geometric_bound_on_output_feedback} and an argument analogous to that used in the previous part, that the stage cost \eqref{eq:stage_cost} is geometrically bounded.
    \newline To prove existence of optimal solutions, we regard the cost functional $J$ as an extended\footnote{The term ``extended'' refers to the fact that some feasible trajectories have infinite cost.} real-valued map defined on the set of feasible sequences $\{(\Cost_t, \Delta_t)\}_{t \in \N}$, equipped with the product topology. We first show that this space is compact. For each $t \in \N$, the pair $(\Cost_t, \Delta_t)$ belongs to the set
    \begin{align*}
        \mathfrak{S}_t(\bar{\Cost}) \coloneqq \{(\Cost, \Delta) : \Cost \in \mathfrak{R}_t(\bar{\Cost}), \Delta \in \mathcal{F}(\Cost)\},
    \end{align*}
    where $\mathfrak{R}_t(\bar{\Cost})$ denotes the reachable set at time $t$ from the initial reward $\bar{\Cost}$. The set
    \begin{align*}
        \mathfrak{S}_0(\bar{\Cost}) = \{\bar{\Cost}\} \times \mathcal{F}(\bar{\Cost})
    \end{align*}
    is clearly compact. Assume that $\mathfrak{S}_t(\bar{\Cost})$ is compact for a generic time $t \in \N$. The reachable set at time $t + 1$ is given by
    \begin{align*}
        \mathfrak{R}_{t+1}(\bar{\Cost}) = \mathcal{L}(\mathfrak{S}_t(\bar{\Cost})),
    \end{align*}
    where $\mathcal{L} : (\Cost, \Delta) \mapsto R + B \Delta$, and is therefore compact by continuity of $\mathcal{L}$. Consider, without loss of generality, a sequence $(\Cost_k, \Delta_k) \in \mathfrak{S}_{t+1}(\bar{\Cost})$ such that $\Cost_k \to \hat{\Cost} \in \mathfrak{R}_{t+1}(\bar{\Cost})$ as $k \to \infty$. Since
    \begin{align*}
        \mathcal{F}(\Cost_k) \subseteq \feasset{\Cost_k} \subset \mathcal{B}_*(\Cost_k) \subseteq \mathcal{B}_*(\bar{\Cost})
    \end{align*}
    for all $k \in \N$, see \eqref{eq:feasible_ball}, there exists a subsequence $(\Cost_{k'}, \Delta_{k'})$ converging to some $\bigl(\hat{\Cost}, \hat{\Delta}\bigr) \in \mathfrak{R}_{t+1}(\bar{\Cost}) \times \mathcal{B}_*(\bar{\Cost})$. Since $\Delta_{k'} \in \mathcal{F}(\Cost_{k'})$ for every $k'$, we have:
    \begin{align*}
        \sum_{i = 1}^m P_i^l \Cost_{k'}^i + \sum_{j = 1}^n Q_j^l \Delta_{k'}^j \leq 0, \quad \forall l \in \{1, \hdots, d\}.
    \end{align*}
    Passing to the limit preserves the inequalities, and it follows by continuity that $\hat{\Delta} \in \mathcal{F}(\hat{\Cost})$. Therefore, $\mathfrak{S}_{t+1}(\bar{\Cost})$ is sequentially compact, and thus compact. By induction $\mathfrak{S}_t(\bar{\Cost})$ is compact for all $t \in \N$, and thus the space of feasible sequences, given by $\prod_{t = 0}^{\infty} \mathfrak{S}_t(\bar{\Cost})$, is compact in the product topology by Tychonoff theorem (see, e.g., \cite[Theorem 13]{kelley2017general}). Furthermore, the cost is lower-semicontinuous, since it is the supremum over $\N$ of its partial sums, which are continuous on this space. Now, under either assumption $(a)$ or $(b)$ there exists a feasible solution with finite cost $\bar{J}$. Hence, the sublevel set $J^{-1}((-\infty, \bar{J}]) = J^{-1}([0, \bar{J}])$ is nonempty. By lower-semicontinuity this set is closed and thus compact as a closed subset of a compact topological space. By the extended Weierstrass theorem for lower-semicontinuous functions, $J$ attains its minimum on this set.
\end{proof} \section{Approximate-Model Reward-Balancing Methods}
\label{sec:section_3}
In this section, we examine the consequences of performing a reward-balancing transformation when only an approximate model is available. Specifically, in Subsection~\ref{subsec:section_3A}, we characterize the effect of an advantage-preserving transformation on the entire class of reward-coupled MDPs. Furthermore, in Subsection~\ref{subsec:section_3B} we discuss how a sample-based normalization strategy can be used to increase the likelihood of achieving optimality in the absence of an exact model. We also propose a principled framework for the design of robust normalizing laws based on a scenario MPC approach, extending the optimal control formulation introduced in the previous section. Finally, in Subsection~\ref{subsec:section_3C} we provide numerical simulations comparing the scenario MPC and the full-output feedback first introduced in~\cite{mustafin2025mdpgeometrynormalizationreward} and recalled in the previous section.
\subsection{Reward Transformations with Approximate Model}
\label{subsec:section_3A}
Let $\mdparg{\cost}{\disc \f}$ be the MDP modeling the problem we aim to solve, and suppose that we have access to an approximation $\hat{\f}$ of the transition function, which could be obtained, for example, from experimental data. Consequently, our reference MDP is actually $\mdparg{\cost}{\disc \hat{\f}}$ and we can normalize its reward function using one of the methods discussed in the previous section. However, at this point, a question naturally arises: what happens to the original MDP? This question is difficult to answer a priori, as it depends heavily on the specific model $\hat{\f}$. Nevertheless, we can reasonably expect that the advantage functions, together with the relative positions of the value functions of $\mdparg{\cost}{\disc \f}$ will be altered during the normalization procedure. This is because we are shifting the reward function of two reward-coupled MDPs along the orbit of one, thus moving it away from the orbit of the other. Clearly if $\hat{\f}$ is a good approximation of~$\f$, we might still expect that, in the end, the greedy policy will be optimal for $\mdparg{\cost}{\disc \f}$. This occurs in the case where the orbit of the approximate MDP intersects the same sections of the normal set as $\mdparg{\cost}{\disc \f}$, as already discussed. This being said, the normalized reward will inevitably differ from the unique normalization of the original MDP. The next proposition addresses part of the problem, by capturing the effect of an advantage-preserving reward transformation on any reward-coupled MDP.
\begin{proposition}\label{prop:modified_transformation}
    Let $\mdparg{\cost}{\disc \hat{\f}}$ and $\mdparg{\cost}{\disc \f}$ be two reward-coupled MDPs, and suppose that the reward function is updated with an advantage-preserving transformation of the first MDP:
    \begin{align*}
        \Costpert = \Cost + \bigl( \disc \Fapx - \Proj \bigr) \Delta.
    \end{align*}
    Then, the value functions of the second MDP change according to the following relations:
    \begin{align}\label{eq:modified_transform}
        \Varg{\policy, \hat{\cost}} = \Varg{\policy, \cost} - \Delta + \disc \bigl(I_n - \disc \Ffb{\policy}\bigr)^{-1} \bigl(\Fapxfb{\policy} - \Ffb{\policy} \bigr) \Delta,
    \end{align}
    for all $\policy \in \Sect$.
\end{proposition}
\begin{proof}
    Consider an arbitrary policy $\policy \in \Sect$. Both the value functions corresponding to $\hat{\cost}$ and $\cost$, respectively, satisfy the corresponding Bellman equation. In vector notation:
    \begin{align}
        0 &= \Costfb{\policy} - (I_n - \disc \Ffb{\policy}) \Varg{\policy, \cost}\label{eq:bellman_value}\\
        0 &= \underbrace{\Costfb{\policy} - (I_n - \disc \Fapxfb{\policy}) \Delta}_{\eqqcolon \Costpert_{\policy}} - (I_n - \disc \Ffb{\policy}) \Varg{\policy, \hat{\cost}}. \label{eq:bellman_value_approx_transf}
    \end{align}
    The comparison between~\eqref{eq:bellman_value} and~\eqref{eq:bellman_value_approx_transf} leads to
    \begin{align}
        \Varg{\policy, \hat{\cost}} = \Varg{\policy, \cost} + (I_n - \disc \Ffb{\policy})^{-1} \bigl(\disc \Fapxfb{\policy} - I_n\bigr) \Delta.
    \end{align}
    Expanding the second term on the right-hand side,
    \begin{align*}
        (I_n - \disc \Ffb{\policy})^{-1} \bigl(\disc \Fapxfb{\policy} - I_n\bigr) &= \left( \sum_{t = 0}^\infty \disc^t (\Ffb{\policy})^t\right) \bigl(\disc \Fapxfb{\policy} - I_n\bigr)\\
        &= \disc (I_n - \disc \Ffb{\policy})^{-1} \Fapxfb{\policy} - I_n - \disc \left(\sum_{t = 0}^\infty \disc^t (\Ffb{\policy})^t\right) \Ffb{\policy}\\
        &= - I_n + \disc (I_n - \disc \Ffb{\policy})^{-1} \bigl(\Fapxfb{\policy} - \Ffb{\policy} \bigr),
    \end{align*}
    from which the proposition follows.
\end{proof}
\begin{center}
    \begin{figure}[!t]
        \centering
        \includegraphics[width = 0.8\textwidth]{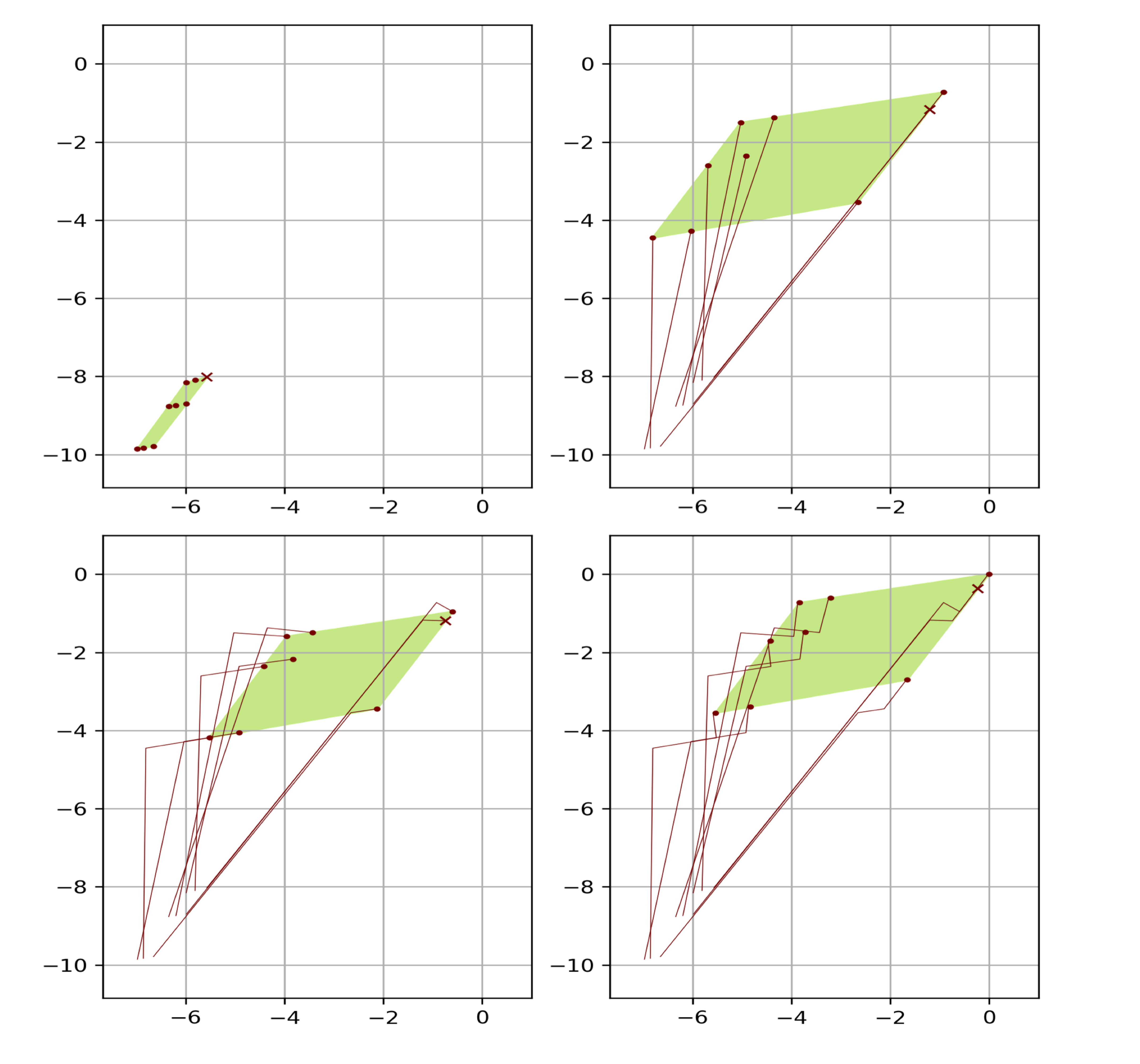}
        \caption{Example of the evolution of the value functions set (green region, including stochastic policies) for the same MDP as in Figure~\ref{fig:value_functions_evolution_exact_model} when the reward is updated using the RB-S control law computed using a wrong model. Each suboptimal value function corresponding to a stationary deterministic policy is marked with a red dot, whereas the optimal one is marked with a cross. Depicted are the original set (top-left) and the set after one (top-right), two (bottom-left) and thirty iterations (bottom-right). In this case a suboptimal value function gets annihilated, precisely the one which is optimal for the model used to perform the reward updates.}
        \label{fig:value_functions_evolution_wrong_model}
    \end{figure}
\end{center}
\noindent The following proposition shows that the additional term in~\eqref{eq:modified_transform} can be bounded by a function of the model mismatch and the input amplitude.
\begin{proposition}
    The deviation of the perturbed update~\eqref{eq:modified_transform} from the exact-model update admits the following upper bound:
    \begin{align}
        \left\|\Varg{\policy, \hat{\cost}} - (\Varg{\policy, \cost} - \Delta) \right\|_2 \leq \varepsilon \sqrt{n} \frac{\disc}{1 - \disc} \|\Delta\|_2,
    \end{align}
    where $n$ is the cardinality of the state space, $\disc$ is the discount factor and $\varepsilon$ is the operator 2-norm of the model error:
    \begin{align*}
        \varepsilon \coloneqq \bigl\|\Fapx - \F\bigr\|_2.
    \end{align*}
\end{proposition}
\begin{proof}
    Let $\policy \in \Sect$ be any policy. From Proposition~\ref{prop:modified_transformation}, it immediately follows that
    \begin{align}\label{eq:partial_bound}
        \twonorm{\Varg{\policy, \hat{\cost}} - (\Varg{\policy, \cost} - \Delta)} &\leq \frac{\disc}{\sigma^{\min}_{\policy, \disc}} \twonorm{\Fapxfb{\policy} - \Ffb{\policy}} \cdot \twonorm{\Delta}\\
        &\leq \frac{\disc}{\sigma^{\min}_{\policy, \disc}} \twonorm{\Pi} \cdot \twonorm{\Fapx - \F } \cdot \twonorm{\Delta}\\
        &= \varepsilon \frac{\disc}{\sigma^{\min}_{\policy, \disc}} \twonorm{\Delta},
    \end{align}
    where $\sigma^{\min}_{\policy, \disc}$ denotes the minimum singular value of $A \coloneqq (I_n - \disc \Ffb{\policy})$. Now, the matrix $A$ is always diagonally dominant, regardless of the policy and the discount factor, since for all $i \in \{1, \hdots, n\}$ we have
    \begin{align*}
        \left|A_i^i\right| = 1 - \disc [\Ffb{\policy}]_i^i > \disc \left(1 - [\Ffb{\policy}]_i^i\right) \stkrl{a}{=} \sum_{j \neq i} \disc [\Ffb{\policy}]_j^i = \sum_{j \neq i} \left|A_j^i\right|,
    \end{align*}
    where (a) follows from the row-stochasticity of $\Ffb{\policy}$. Therefore, by \cite[Theorem 1]{VARAH19753}, it satisfies
    \begin{align*}
        \infnorm{A^{-1}} < \frac{1}{\alpha}
    \end{align*}
    with
    \begin{align*}
        \alpha &\coloneqq \min_i \left\{\left|A_i^i\right| - \sum_{j \neq i} \left|A_j^i\right|\right\}\\
        &= \min_i \left\{1 - \disc [\Ffb{\policy}]_i^i - \sum_{j \neq i} \disc [\Ffb{\policy}]_j^i\right\}\\
        &= \min_i \left\{1 - \sum_{j = 1}^n \disc [\Ffb{\policy}]_j^i\right\}\\
        &= \min_i \left\{1 - \disc \right\}\\
        &= 1 - \disc.
    \end{align*}
    Finally, norm equivalence implies
    \begin{align*}
        \frac{1}{\sigma^{\min}_{\policy, \disc}} &= \twonorm{A^{-1}} \leq \sqrt{n} \infnorm{A^{-1}} < \frac{\sqrt{n}}{1 - \disc}
    \end{align*}
    which, used in \eqref{eq:partial_bound}, concludes the proof.
\end{proof}

\subsection{Reward-Balancing and Model Sampling}
\label{subsec:section_3B}
The previous results characterize the structural relationship between a perturbed, advantage-preserving update and the value functions of the true model. We now turn to analyzing the long term behavior of the normalization process under model uncertainties. As already discussed, choosing a single approximate model completely determines the outcome of the normalization process, and no choice of input can change this outcome. Consequently, in situations where a certain level of model accuracy cannot be guaranteed, a different approach is required. Suppose we have access to random realizations $\Fapx_t$ of the true model $\F$. Then the reward can be updated as
\begin{subequations}\label{eq:random_reward_update}
    \begin{align}
        \Costpert_{t + 1} &= \Costpert_t + \Bapx_t \Delta_t\\
        &= \Costpert_t + \B \Delta_t + \bigl(\Bapx_t - \B\bigr)\Delta_t,
    \end{align}
\end{subequations}
where
\begin{align}\label{eq:B_sample_def}
    \Bapx_t \coloneqq \disc \Fapx_t - \Proj,
\end{align}
which falls into the standard framework of stochastic systems with multiplicative noise (see, e.g., \cite{gravell2020robust, coppens2020data}). In this setting, the outcome of a normalization process is not predetermined by a single fixed model. A first drawback of this approach is that, in general, the admissible set changes at every model realization. It is therefore convenient to identify a model-invariant subset common to \emph{every} admissible set. For example, the set $\{\Delta : \outmap(\Cost) \leq \Delta \leq 0\}$ is a subset of $\feasset{\Cost}$ for all $\f$ and is therefore independent of the particular model choice. In general, however, this set is not the largest subset of $\feasset{\Cost}$ with this invariance property, as established by the following theorem.
\begin{theorem}[\textbf{The largest model-invariant admissible set}]\label{th:largest_invariant_set}
    For any nonpositive reward function $\cost$ and discount factor $\disc \in [0, 1)$ the set
    \begin{align}\label{eq:feasset_def}
        \invset{\Cost} \coloneqq \bigcap_{i,k = 1}^n \bigl\{ \Delta \in \R^n : 0 \geq \Delta^i \geq \disc \Delta^k + \Cost^j, ~ j \in \U_i\bigr\} %
    \end{align}
    satisfies~\eqref{eq:feasibility_of_general_constraint} for any model $\f$. In particular, it is the maximal model-independent feasible set in the following sense:
    \begin{align}
        \invset{\Cost} = \bigcap_{\f} \feasset{\Cost}.
    \end{align}
    Furthermore, it satisfies condition (b) of Theorem~\ref{th:sufficient_conditions_feasibility}, i.e.:
    \begin{align}
        \outmap(\Cost) \in \invset{\Cost}, \quad \forall \Cost \leq 0.
    \end{align}
\end{theorem}
\begin{proof}
    Consider any $\Delta \in \invset{\Cost}$ and choose any model $\Fapx$. We need to show that $\Delta \in \feasset{\Cost}$, i.e., by~\eqref{eq:alternative_definition_feasible_set}, that $\mathcal{L}^\Delta(\Cost) \leq 0$ (the other condition, $\Delta \leq 0$, is satisfied by definition). Let $j \in \{1, \hdots, m\}$ be any action index and $i = \proj(j)$ be the corresponding state index. Then, we have:
    \begin{align*}
        \bigl[\mathcal{L}^\Delta(\Cost)\bigr]^j &= \Cost^j + (\disc\Fapx_i^j - 1)\Delta^i + \sum_{k \neq i} \disc\Fapx_k^j \Delta^k\\
        &= \Cost^j - \Delta^i + \disc \sum_{k = 1}^n \Fapx_k^j \Delta^k\\
        &\stkrl{a}{=} - \Delta^i + \sum_{k = 1}^n \Fapx_k^j \bigl(\disc \Delta^k + \Cost^j\bigr)\\
        &\stkrl{b}{\leq} -\Delta^i + \underbrace{\sum_{k = 1}^n \Fapx_k^j}_{= 1} \Delta^i\\
        &= 0,
    \end{align*}
    where $(a)$ follows from the row-stochasticity of $\Fapx$, whereas $(b)$ follows from the definition~\eqref{eq:feasset_def}. Since $j$ is arbitrary, this shows that $\mathcal{L}^\Delta(\Cost) \leq 0$, thus concluding the first part of the proof. In general, some of the linear inequalities defining the set $\invset{\Cost}$ could be redundant: we say that a triple $(i, j, k)$ is active if the corresponding inequality $\Delta^i \geq \disc \Delta^k + \Cost^j$ is non-redundant, i.e., we cannot remove it without changing the set. Observe that the inequality corresponding to any active triple $(i, j, k)$ coincides with the inequality $[\mathcal{L}_\Delta(\Cost)]^j \leq 0$, for a model $\Fapx$ such that $\Fapx_k^j = 1$ and $\Fapx_l^j = 0$ for all $l \neq k$. Indeed:
    \begin{align*}
        0 \geq \Cost^j - \Delta^i + \sum_{l = 1}^n \disc \Fapx_l^j \Delta^l = \Cost^j - \Delta^i + \disc \Delta^k,
    \end{align*}
    which shows that the inequality corresponding to the active triple $(i, j, k)$ is indeed part of the definition of the feasible set for this specific model, and this concludes the proof of the second part. Finally, if $\Cost \leq 0$, for all index tuples $(i, j, k)$, with $j\in \U_i$, it holds:
    \begin{align*}
        0 \geq \max\nolimits_{l \in \U_i} \Cost^l \geq \Cost^j \geq \Cost^j + \underbrace{\disc \max\nolimits_{l \in \U_k} \Cost^l}_{\leq 0},
    \end{align*}
    and thus $\outmap(\Cost) \in \invset{\Cost}$, which is precisely condition (b) of Theorem~\ref{th:sufficient_conditions_feasibility}.
\end{proof}
\noindent Theorem~\ref{th:largest_invariant_set} guarantees that, if we have an initial reward $\Costpert_0 \leq 0$ and we apply~\eqref{eq:random_reward_update} with a control input $\Delta_t \in \invset{\Costpert_t}$, then $\Costpert_t$ remains nonpositive for all $t$, regardless of the specific model realizations $\{\Bapx_t\}_{t \in \N}$. However, we have no guarantee that the reward vector can be asymptotically normalized via~\eqref{eq:random_reward_update}. The following theorem addresses this issue by providing a sufficient condition for a control law to preserve the nonpositivity of the reward vector and to make $\Costpert_t$ converge to the normal set, regardless of the sequence of models used for the update.
\begin{theorem}[\textbf{Sufficient condition for robust normalization}]\label{th:sufficient_condition_for_robust_normalization}
    If an input control law $h : \R^m \to \R^n$ satisfies
    \begin{subequations}\label{eq:suff_cond_normalization_stochastic}
        \begin{align}
            &h(\Cost) \in \invset{\Cost}\label{eq:suff_cond_normalization_stochastic_1}\\
            &h^i(\Cost) + \disc \infnorm{h(\Cost)} \leq \outmap^i(\Cost) + \alpha \infnorm{\outmap(\Cost)}, \label{eq:suff_cond_normalization_stochastic_2}
        \end{align}
    \end{subequations}
    for all $i \in \{1, \hdots, n\}$, $\Cost \leq 0$ and for some $\alpha \in [0, 1)$, then $\Delta_t = h(\Costpert_t)$ is geometrically normalizing for the stochastic update~\eqref{eq:random_reward_update}, regardless of the realization $\{\Bapx_t\}_{t \in \N}$ and the initial condition $\Costpert_0 \leq 0$.
\end{theorem}
\begin{proof}
    Let $\Delta_t = h(\Costpert_t)$ satisfy~\eqref{eq:suff_cond_normalization_stochastic}. Then, for any $i \in \{1, \hdots, n\}$ we have:
    \begin{align}
        \notag\left| \outmap^i\bigl(\Costpert_{t + 1}\bigr) \right| &\stkrl{a}{=} - \max_{j \in \U_i} \Costpert_{t+1}^j\\
        &\notag= - \max_{j \in \U_i}\left\{ \Costpert_t^j + \disc \sum_{k = 1}^n \bigl[\Fapx_t\bigr]_k^j \Delta_t^k - \Delta_t^i \right\}\\
        &\notag= \Delta_t^i + \min_{j \in \U_i} \left\{ - \Costpert_t^j + \disc \sum_{k = 1}^n \bigl[\Fapx_t\bigr]_k^j \left(-\Delta_t^k \right)\right\}\\
        &\notag\stkrl{b}{\leq} \Delta_t^i + \min_{j \in \U_i} \left\{ - \Costpert_t^j + \disc \max_k\left\{ -\Delta_t^k \right\}\right\}\\
        &\notag= \Delta_t^i + \min_{j \in \U_i} \left\{ - \Costpert_t^j \right\} + \disc \max_k\left\{ -\Delta_t^k \right\}\\
        &\notag\stkrl{c}{=} \Delta_t^i - \max_{j \in \U_i} \left\{ \Costpert_t^j \right\} + \disc \max_k\left\{ \bigl| \Delta_t^k \bigr| \right\}\\
        &\notag= \Delta_t^i - \outmap^i\bigl(\Costpert_t\bigr) + \disc \infnorm{\Delta_t}\\
        &\stkrl{d}{\leq} \alpha\infnorm{\outmap\bigl(\Costpert_t\bigr)},\label{eq:lyapunov_bound_inequality}
    \end{align}
    where (a) holds since~\eqref{eq:suff_cond_normalization_stochastic_1} ensures sign preservation of the reward and thus $\Costpert_0 \leq 0 \implies \outmap(\Costpert_t) \leq 0$ for all $t$, (b) holds since any realization $\Fapx_t$ is row-stochastic, (c) holds since~\eqref{eq:suff_cond_normalization_stochastic_1} implies $\Delta_t \leq 0$, and (d) follows from~\eqref{eq:suff_cond_normalization_stochastic_2}. Since~\eqref{eq:lyapunov_bound_inequality} holds for all $i$, we have:
    \begin{align}
        \infnorm{\outmap\bigl(\Costpert_{t + 1}\bigr)} \leq \alpha \infnorm{\outmap\bigl(\Costpert_{t}\bigr)},
    \end{align}
    and thus $\Costpert_t$ asymptotically converges to the normal set by Theorem \ref{prop:nec_suff_condition_for_normality}.
\end{proof}
\noindent The following corollary shows that one of the control laws analyzed in Subsection~\ref{subsec:section_2B} for the deterministic case satisfies the condition of Theorem~\ref{th:sufficient_condition_for_robust_normalization}, thereby guaranteeing normalization also in the stochastic setting.
\begin{corollary}\label{cor:full_feedback_is_robust}
    The full-output feedback $\outmap$, introduced in Definition \ref{def:MFRBS}, satisfies \eqref{eq:suff_cond_normalization_stochastic} for all $\alpha \in [\disc, 1)$, and thus it is geometrically normalizing in the stochastic case.
\end{corollary}
\begin{proof}
    By Theorem \ref{th:largest_invariant_set}, the full-output feedback satisfies~\eqref{eq:suff_cond_normalization_stochastic_1}, while it clearly satisfies~\eqref{eq:suff_cond_normalization_stochastic_2} with $\alpha \geq \disc$, since replacing $h$ with $\outmap$ yields $\disc \infnorm{\outmap(\Costpert)} \leq \alpha \infnorm{\outmap(\Costpert)}$.
\end{proof}
\begin{center}
    \begin{figure}[!t]
        \centering
        \includegraphics[width = 0.8\textwidth]{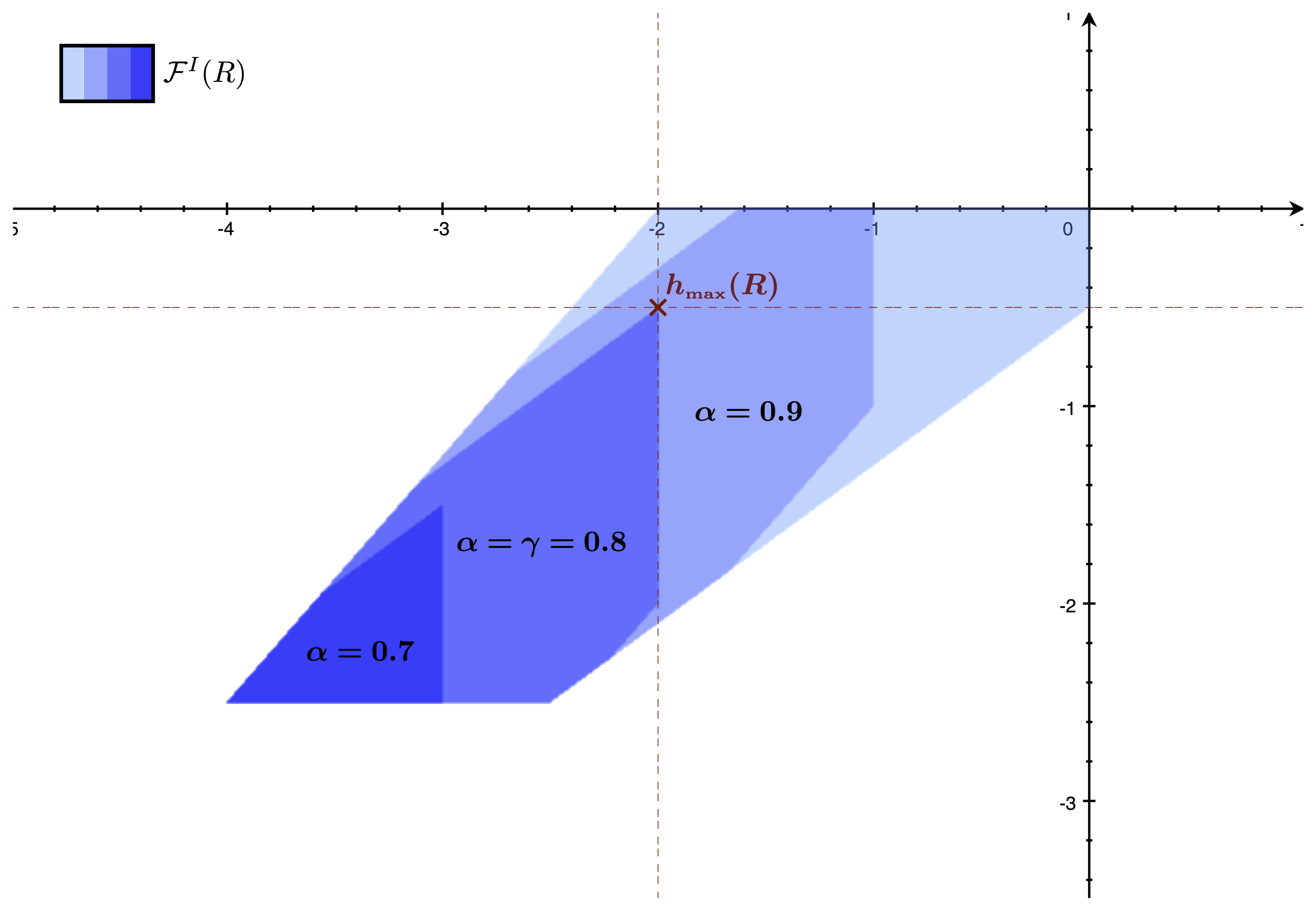}
        \caption{The largest model-invariant admissible set $\invset{\Cost}$ for a reward vector $\Cost$ with $\outmap(\Cost) = (-2, -0.5)$ and a discount factor $\disc = 0.8$. The three darker, nested subsets of $\invset{\Cost}$ represent the regions satisfying~\eqref{eq:suff_cond_normalization_stochastic} for $\alpha \in \{0.7, 0.8, 0.9\}$.}
        \label{fig:largest_invariant_set}
    \end{figure}
\end{center}
\noindent Figure~\ref{fig:largest_invariant_set} illustrates an example of an invariant set, highlighting the subsets that satisfy the condition of Theorem~\ref{th:sufficient_condition_for_robust_normalization}. Observe that, in general, $\alpha$ cannot be taken arbitrarily close to $0$, but, as shown by Corollary \ref{cor:full_feedback_is_robust}, has a lower bound $\bar{\alpha} \in (0, \disc]$. Theorems~\ref{th:largest_invariant_set}-\ref{th:sufficient_condition_for_robust_normalization} provide a framework to guarantee the convergence of the stochastic update~\eqref{eq:random_reward_update} to the normal set. However, it is still necessary to quantify the deviation between the stochastic update and the ideal update, which involves the true average model $\B$. Specifically, for an initial reward $\Cost_0 \leq 0$, we need to analyze and compare the following closed-loop dynamics:
\begin{subequations}\label{eq:ideal_and_perturbed_updates}
    \begin{align}
        \Cost_{t + 1} &= \Cost_t + \B h(\Cost_t)\\
        \Costpert_{t + 1} &= \Costpert_t + \Bapx_t h\bigl(\Costpert_t\bigr), \quad \Costpert_0 = \Cost_0,
    \end{align}
\end{subequations}
where $\Bapx_t$ is defined as in \eqref{eq:B_sample_def}, and $h : \R^m \to \R^n$ is a control law that ensures robust normalization, as stated in Theorem~\ref{th:sufficient_condition_for_robust_normalization}. This guarantees that both reward sequences converge to the normal set. To quantify the discrepancy between the two reward trajectories, we compute the deviation:
\begin{align}
    \hspace{-0.2cm}\Costpert_t - \Cost_t &= \sum_{\tau = 0}^{t - 1} \Bigl( \B h\bigl(\Costpert_\tau\bigr) + \bigl[\Bapx_\tau - \B\bigr] h\bigl(\Costpert_\tau\bigr) - \B h(\Cost_\tau) \Bigr) \notag\\
    &= \sum_{\tau = 0}^{t-1} \bigl[\Bapx_\tau - \B\bigr] h\bigl(\Costpert_\tau\bigr) + \B \sum_{\tau = 0}^{t-1} \bigl[h\bigl(\Costpert_\tau\bigr) - h(\Cost_\tau)\bigr] \label{eq:reward_deviation},
\end{align}
for all $t > 0$. The decomposition~\eqref{eq:reward_deviation} highlights that the deviation arises from two distinct contributions, which we will characterize in the upcoming results. The following lemma establishes their convergence, by showing that any input satisfying the conditions of Theorem~\ref{th:sufficient_condition_for_robust_normalization} decays geometrically fast, thereby ensuring boundedness of the overall deviation~\eqref{eq:reward_deviation}.
\begin{lemma}[\textbf{Geometric bound for robust normalizing controls}]\label{lem:geometric_input_convergence}
    If a feedback control law $h : \R^m \to \R^n$ satisfies conditions~\eqref{eq:suff_cond_normalization_stochastic} of Theorem~\ref{th:sufficient_condition_for_robust_normalization} for some $\alpha \in [0, 1)$, then:
    \begin{align}
        \infnorm{h\bigl(\Costpert_t\bigr)} \leq \frac{\alpha^t}{1 - \disc} \infnorm{\outmap\bigl(\Costpert_0\bigr)}.
    \end{align}
\end{lemma}
\begin{proof}
    Since $h(\Costpert_t) \in \invset{\Costpert_t}$, $\Delta_t = h(\Costpert_t)$ satisfies:
    \begin{align*}
        0 \geq \Delta_t^i \geq \disc \Delta_t^k + \Costpert_t^j,
    \end{align*}
    for all $i, k \in \{1, \hdots, n\}$ and $j \in \U_i$. Choosing $k = i$ and $j \in \argmax_{l \in \U_i} \Costpert_t^l$, yields
    \begin{align*}
        (1 - \disc) \Delta_t^i \geq \max_{l \in \U_i} \Costpert_t^l = \outmap^i\bigl(\Costpert_t\bigr).
    \end{align*}
    By Theorem \ref{th:sufficient_condition_for_robust_normalization} we have
    \begin{align*}
        \infnorm{\outmap\bigl(\Costpert_t\bigr)} \leq \alpha^t \infnorm{\outmap\bigl(\Costpert_0\bigr)},
    \end{align*}
    and thus, for all $i \in \{1, \hdots, n\}$:
    \begin{align*}
        (1 - \disc) \bigl|\Delta_t^i\bigr| \leq \bigl|\outmap^i\bigl(\Costpert_t\bigr)\bigr| \leq \infnorm{\outmap^i\bigl(\Costpert_t\bigr)} \leq \alpha^t \infnorm{\outmap^i\bigl(\Costpert_0\bigr)}.
    \end{align*}
\end{proof}
The following lemma characterizes the stochastic behavior of the first term in \eqref{eq:reward_deviation}, allowing us to derive probabilistic bounds on the overall deviation that explicitly incorporate prior knowledge of the quality of the models $\Bapx_t$.
\begin{lemma}\label{lem:martingale_properties}
    Suppose $\{\Fapx_t\}_{t \in \N}$ is a sequence of i.i.d. random realizations of $\F$ such that $\E{}{\Fapx_t} = \F$ and $\infnorm{\Fapx_t - \F} \leq \mu$ for all $t$, and that the reward is updated according to~\eqref{eq:random_reward_update}-\eqref{eq:B_sample_def}. If the input $\Delta_t$ does not depend on $\Fapx_t$, then the stochastic process $\{M_t\}_{t \in \N}$, defined as
    \begin{subequations}\label{eq:martingale_def}
        \begin{align}
            M_0 &\coloneqq 0\\
            M_t &\coloneqq \sum_{\tau = 0}^{t-1} \bigl(\Bapx_\tau - \B\bigr)\Delta_\tau, \quad t > 0,
        \end{align}
    \end{subequations}
    is a martingale. Furthermore, if $\{\Delta_t\}_{t \in \N}$ is square-summable, then $M_t$ satisfies
    \begin{align}
        \prob\left(\sup_{t\in\N}\infnorm{M_t} \geq \varepsilon\right) \leq 2 \exp\left(-\frac{\varepsilon^2}{2\mu^2\disc^2\sum_{\tau = 0}^\infty \infnorm{\Delta_\tau}^2}\right),
    \end{align}
    for all $\varepsilon > 0$.
\end{lemma}
\begin{proof}
    By the definitions of $\B$ and $\Bapx_t$ (see \eqref{eq:input_matrix}-\eqref{eq:B_sample_def}), we have
    \begin{align}\label{eq:equivalent_def_martingale}
        M_t = \sum_{\tau = 0}^{t-1} \bigl(\Bapx_\tau - \B\bigr)\Delta_\tau = \disc \sum_{\tau = 0}^{t-1} \bigl(\Fapx_\tau - \F\bigr)\Delta_\tau,
    \end{align}
    for all $t > 0$. Let $\mathfrak{F}_t$ be the filtration
    \begin{align}\label{eq:azuma_hoeffding}
        \mathfrak{F}_t \coloneqq \sigma\bigl(\Fapx_0, \hdots, \Fapx_{t-1}\bigr).
    \end{align}
    Then:
    \begin{align*}
        \E{}{M_{t+1} \mid \mathfrak{F}_t} &= \E{}{ \disc\sum_{\tau = 0}^t \bigl(\Fapx_\tau - \F\bigr) \Delta_\tau \mid \mathfrak{F}_t}\\
        &= \E{}{ \disc \sum_{\tau = 0}^{t-1} \bigl(\Fapx_\tau - \F\bigr) \Delta_\tau + \disc\bigl(\Fapx_t - \F\bigr) \Delta_t \mid \mathfrak{F}_t}\\
        &= \E{}{M_t +  \disc\bigl(\Fapx_t - \F\bigr) \Delta_t \mid \mathfrak{F}_t}\\
        &\stkrl{a}{=} M_t +  \disc\E{}{\bigl(\Fapx_t - \F\bigr)} \Delta_t\\
        &= M_t,
    \end{align*}
    where $(a)$ follows from the fact that $\Delta_t$ is independent of $\Fapx_t$ and is thus $\mathfrak{F}_t$-measurable. Therefore, $M_t$ is a $\R^m$-valued martingale, whose martingale increments are bounded as follows
    \begin{align}
        \hspace{-0.3cm}\infnorm{M_t - M_{t-1}} = \disc \infnorm{\bigl(\Fapx_{t - 1} - \F\bigr) \Delta_{t - 1}} \leq \mu \disc \infnorm{\Delta_{t - 1}}\label{eq:increment_bound}.
    \end{align}
    Moreover, square-summability of $\{\Delta_t\}_{t \in \N}$ implies
    \begin{align}\label{eq:square_summability_bound}
        \sum_{t = 1}^\infty \infnorm{M_t - M_{t-1}}^2 \leq \mu^2 \disc^2 \sum_{t = 1}^\infty \infnorm{\Delta_{t-1}}^2 < \infty.
    \end{align}
    Finally, as a Hilbert space under the Euclidean inner product, $(\R^m, \twonorm{\cdot})$ is a $(2, D)$-smooth separable Banach space with $D = 1$, which, together with \eqref{eq:square_summability_bound}, allows us to apply the maximal inequality for 2-smooth Banach spaces~\cite[Theorem 3.5]{pinelis1994optimum}, yielding
    \begin{align*}
        \prob\left(\sup_{t\in\N}\twonorm{M_t} \geq \varepsilon\right) \leq 2 \exp\left(-\frac{\varepsilon^2}{2\mu^2\disc^2\sum_{\tau = 0}^\infty \infnorm{\Delta_\tau}^2}\right),
    \end{align*}
    for all $\varepsilon > 0$. Since, by norm equivalence, $\infnorm{M_t} \leq \twonorm{M_t}$ for all $t \in \N$, we have
    \begin{align*}
        \sup_{t \in \N} \infnorm{M_t} \leq \sup_{t \in \N} \twonorm{M_t}.
    \end{align*}
    This implies the following event inclusion
    \begin{align*}
        \left\{\sup_{t \in \N} \infnorm{M_t} \geq \varepsilon\right\} \subseteq \left\{\sup_{t \in \N} \twonorm{M_t} \geq \varepsilon\right\},
    \end{align*}
    which leads to
    \begin{align*}
        \prob\left(\sup_{t \in \N} \infnorm{M_t} \geq \varepsilon\right) \leq \prob\left(\sup_{t \in \N} \twonorm{M_t} \geq \varepsilon\right).
    \end{align*}
\end{proof}
\begin{remark}
    Since $\Costpert_t$ is $\mathfrak{F}_t$-measurable, the previous lemma also holds for a feedback control law of the form $\Delta_t = h(\Costpert_t)$, provided that $h$ is independent of $\Fapx_t$.
\end{remark}
\begin{remark}
    Observe that the model mismatch $\mu$ is bounded by
    \begin{align*}
        0 \leq \mu \leq 2,
    \end{align*}
    since, if $A_1, A_2 \in \R^{m\times n}$ are row-stochastic, then
    \begin{align*}
        \infnorm{A_1 - A_2} \leq \infnorm{A_1} + \infnorm{A_2} = 2.
    \end{align*}
\end{remark}
\noindent With the previous lemmas at hand, we are now ready to derive a probability bound on the deviation~\eqref{eq:reward_deviation}.
\begin{theorem}[\textbf{Probability bound on the stochastic deviation}]\label{th:probability_bound_deviation}
    Let $\{\Fapx_t\}_{t \in \N}$ be a sequence of i.i.d. random realizations of $\F$ such that $\E{}{\Fapx_t} = \F$ and $\infnorm{\Fapx_t - \F} \leq \mu$, and let $h : \R^m \to \R^n$ be a model-independent feedback control law satisfying~\eqref{eq:suff_cond_normalization_stochastic} for some $\alpha \in [0, 1)$ and for all $\Cost \leq 0$. Let $\{\Cost_t\}_{t \in \N}$ and $\{\Costpert_t\}_{t \in \N}$ be the nominal and perturbed updates, respectively solutions of~\eqref{eq:ideal_and_perturbed_updates}. Then:
    \begin{align}
        \prob\left(\sup_{t \in \N} \infnorm{\Costpert_t - \Cost_t} \geq \varepsilon\right) \leq 2 \exp\left(-c \frac{(\varepsilon - \threshold)^2}{\infnorm{\outmap(\Cost_0)}^2}\right),
    \end{align}
    for all $\varepsilon \geq \threshold$, where
    \begin{subequations}
        \begin{align}
        \threshold &\coloneqq 2\frac{1 + \disc}{(1 - \alpha)(1 - \disc)} \infnorm{\outmap(\Cost_0)}\label{eq:base_threshold}\\
        c &\coloneqq \frac{(1 - \alpha^2)(1 - \disc)^2}{2\mu^2\disc^2}.
    \end{align}
    \end{subequations}
\end{theorem}
\begin{proof}
    Leveraging \eqref{eq:reward_deviation}, we have, for all $t > 0$:
    \begin{align}
        \infnorm{\Costpert_t - \Cost_t} &= \infnorm{\sum_{\tau = 0}^{t-1} \bigl[\Bapx_\tau - \B\bigr] h\bigl(\Costpert_\tau\bigr) + \B \sum_{\tau = 0}^{t-1} \bigl[h\bigl(\Costpert_\tau\bigr) - h(\Cost_\tau)\bigr]} \notag\\
        &\stkrl{1}{\leq} \infnorm{M_t} + (1 + \disc) \sum_{\tau = 0}^{t-1} \infnorm{h\bigl(\Costpert_\tau\bigr) - h(\Cost_\tau)} \notag\\
        &\leq \infnorm{M_t} + (1 + \disc) \sum_{\tau = 0}^{t-1} \left(\infnorm{h\bigl(\Costpert_\tau\bigr)} + \infnorm{h(\Cost_\tau)}\right) \notag\\
        &\stkrl{2}{\leq} \infnorm{M_t} + 2\frac{1 + \disc}{1 - \disc} \infnorm{\outmap(\Cost_0)}\sum_{\tau = 0}^{t - 1} \alpha^\tau \label{eq:improvable_bound}\\
        &\leq \infnorm{M_t} + \underbrace{2\frac{1 + \disc}{(1 - \alpha)(1 - \disc)} \infnorm{\outmap(\Cost_0)}}_{= \threshold} \notag,
    \end{align}
    where in (1) we used \eqref{eq:martingale_def} and the fact that
    \begin{align*}
        \infnorm{\B} &= \max_j\left\{\sum_{k = 1}^n\bigl|\B_k^j\bigr|\right\}\\
        &= \max_j\left\{\bigl(1 - \disc \F_{\scalebox{0.6}{$\proj(j)$}}^j\bigr) + \disc\bigl(1 - \F_{\scalebox{0.6}{$\proj(j)$}}^j\bigr)\right\}\\
        &= \max_j\left\{1 + \disc - 2 \disc \F_{\scalebox{0.6}{$\proj(j)$}}^j\right\}\\
        &\leq 1 + \disc,
    \end{align*}
    while (2) follows from Lemma \ref{lem:geometric_input_convergence}. Hence, since for all $\varepsilon \geq \threshold$:
    \begin{align}
        \left\{ \sup_{t \in \N} \infnorm{\Costpert_t - \Cost_t} \geq \varepsilon \right\} \subseteq \left\{ \sup_{t \in \N} \infnorm{M_t} \geq \varepsilon - \threshold \right\},
    \end{align}
    and since $\{h(\Costpert_t)\}_{t \in \N}$ is square-summable (in fact, geometrically bounded), by Lemma \ref{lem:martingale_properties} we have:
    \begin{align*}
        \prob\left(\sup_{t \in \N} \infnorm{\Costpert_t - \Cost_t} \geq \varepsilon\right) \leq 2 \exp\left(-\frac{(\varepsilon - \threshold)^2}{2\mu^2\disc^2\sum_{\tau = 0}^\infty \infnorm{h(\Costpert_\tau)}^2}\right).
    \end{align*}
    The result follows since, by Lemma \ref{lem:geometric_input_convergence}:
    \begin{align*}
        \sum_{\tau = 0}^\infty \infnorm{h(\Costpert_\tau)}^2 &\leq \sum_{\tau = 0}^\infty \frac{\alpha^{2\tau}}{(1 - \disc)^2} \infnorm{\outmap(\Cost_0)}^2 = \frac{\infnorm{\outmap(\Cost_0)}^2}{(1 - \disc)^2(1 - \alpha^2)}.
    \end{align*}
\end{proof}
\noindent Theorem~\ref{th:probability_bound_deviation} provides an exponential bound on the probability that the deviation~\eqref{eq:reward_deviation} exceeds a given threshold $\varepsilon$, relative to a base threshold $\threshold$ determined by the input decay. Notably, a smaller model mismatch $\mu$ yields a faster exponential decay. Moreover, the base threshold can potentially be further reduced if additional information about the feedback control law is available, as discussed in the following remark.
\begin{remark}
    Observe that the term $1 - \disc$ at the denominator in the definition of $\threshold$ arises from the application of Lemma~\ref{lem:geometric_input_convergence} in~\eqref{eq:improvable_bound}. This bound relies only on the fact that the feedback control law satisfies~\eqref{eq:suff_cond_normalization_stochastic}, making it fairly conservative. Consequently, the bound may be tightened if additional information about the specific control law is available. In particular, if $h = \outmap$, the term $1 - \disc$ in~\eqref{eq:base_threshold} can be replaced by~$1$, yielding:
    \begin{align*}
        \threshold &= 2\frac{1 + \disc}{1 - \alpha} \infnorm{\outmap(\Cost_0)},
    \end{align*}
    where $\alpha$ can be set to $\disc$, by Corollary \ref{cor:full_feedback_is_robust}.
\end{remark}
\noindent We conclude the section by leveraging the optimal control framework introduced in Subsection~\ref{subsec:section_2C} to design a normalizing feedback control law. This framework allows us to exploit both constraints and the cost function as degrees of freedom for the design of a control law for the stochastic dynamics~\eqref{eq:random_reward_update}. Ideally, we aim to design a robust normalizing control law, for instance, one satisfying the conditions of Theorem~\ref{th:sufficient_condition_for_robust_normalization}. Condition~\eqref{eq:suff_cond_normalization_stochastic_1} can be directly enforced as a constraint, since it can be expressed as a finite set of linear inequalities of the form~\eqref{eq:general_feas_constraint}. On the other hand, instead of explicitly enforcing~\eqref{eq:suff_cond_normalization_stochastic_2}, we adopt a scenario-based approach. Specifically, we consider a finite number of realizations (scenarios) of the stochastic dynamics over a sufficiently long prediction horizon, and we seek a given control input that minimizes the average cumulative output defined in~\eqref{eq:output_regulation_implicit}, in analogy with Problem~\ref{eq:ideal_optimal_control_problem}. This leads to a scenario-based MPC formulation, as described in~\cite[Section 3.2]{schildbach2014scenario}. Given a reward vector $\Cost_t$ at time $t$, we consider the following optimal control problem
\begin{subequations}\label{eq:scenario_optimal_control_problem}
    \begin{align}
        &\min_{\sml \{(\Cost_{\sigma,\tau}, \Delta_\tau)\}} \frac{1}{N} \sum_{\sigma = 1}^N \sum_{\tau = t}^{t+T-1} \left(\sum_{i = 1}^n -\max_{j \in \U_i} \Cost_{\sigma, \tau+1}^j\right) + \epsilon \twonorm{\Delta_\tau}^2\\\notag
        &\quad\text{subj.to:}\\
        &\qquad \Cost_{\sigma, t} = \Cost_t \leq 0,\\
        &\qquad \Cost_{\sigma, \tau+1} = \Cost_{\sigma,\tau} + \Bapx_{\sigma,\tau} \Delta_\tau,\\
        &\qquad \Delta_\tau \in \invset{\Cost_{\sigma,\tau}},\\
        &\quad\text{for all } \tau \in \{t, \hdots,t+T-1\}, \sigma \in \{1, \hdots, N\}.\notag
    \end{align}
\end{subequations}
Here, $N$ is the number of sampled scenarios, while the regularization term $\epsilon \twonorm{\Delta_\tau}^2$ promotes square-summability of the control input. At each iteration $t$, we sample $N$ independent realizations and solve problem~\eqref{eq:scenario_optimal_control_problem}. The first input solution $\Delta_t$, which depends on the current reward $\Cost_t$, is then used to update the reward with newly sampled model $\Bapx_t$, and the process is repeated in a receding horizon fashion. This approach yields a feasible, model-invariant feedback control law, which can be compared with the full-output feedback strategy of Definition~\ref{def:MFRBS}. The latter has been selected because it is the only control law introduced in Section~\ref{subsec:section_2B} that satisfies the conditions of Theorem~\ref{th:sufficient_condition_for_robust_normalization} (cf. Corollary \ref{cor:full_feedback_is_robust}). In the following subsection, we present a numerical example in which the MPC feedback law derived from~\eqref{eq:scenario_optimal_control_problem} outperforms this strategy, originally proposed in~\cite{mustafin2025mdpgeometrynormalizationreward}. 
\subsection{Numerical example}
\label{subsec:section_3C}
Consider the MDP consisting of two states $\{\state_1, \state_2\}$ with five actions each $\{\act_{ij}\}_{j = 1}^5$, with transition probabilities to $\state_1$
\begin{center}
    \begin{tabular}{|c|c|c|c|c|c|}
        \hline
        $\f(\state_1, \act_{ij})$ & $\act_{i1}$ & $\act_{i2}$ & $\act_{i3}$ & $\act_{i4}$ & $\act_{i5}$\\
        \hline
        $i = 1$ & 0.35 & 0.4  & 0.7 & 0.3 & 0.4 \\
        $i = 2$ & 0.75 & 0.8 & 0.2 & 0.55 & 0.25  \\
        \hline
    \end{tabular}
\end{center}
and with reward function
\begin{center}
    \begin{tabular}{|c|c|c|c|c|c|}
        \hline
        $\cost(\act_{ij})$ & $\act_{i1}$ & $\act_{i2}$ & $\act_{i3}$ & $\act_{i4}$ & $\act_{i5}$\\
        \hline
        $i = 1$ & -0.7 & -0.5  & -0.9 & -0.45 & -0.2 \\
        $i = 2$ & -3.2 & -3 & -2.8 & -3    & -2.9  \\
        \hline
    \end{tabular}
\end{center}
and a discount factor $\disc = 0.8$. For this MDP the (unique) optimal policy $\policy^*$ is the given by:
\begin{align*}
    \policy^*(\state_1) &= \act_{15}\\
    \policy^*(\state_2) &= \act_{22}.
\end{align*}
We set a scenario MPC based on~\eqref{eq:scenario_optimal_control_problem} with a prediction horizon $T = 20$ and a regularization parameter $\epsilon = 2$. The model samples $\Bapx_t$ are drawn from a uniform distribution around the true model, with a maximum entry-wise deviation of 0.2. Both the full output feedback and the scenario MPC were evaluated over $500$ Monte-Carlo simulations. For the scenario MPC, the number of scenarios considered varies between 1 and 4. Each simulation is interrupted when
\begin{align*}
    \infnorm{y_t} = \infnorm{\outmap(\Cost_t)} < 10^{-3}.
\end{align*}
At the end of each simulation, the resulting greedy policy was compared with the true optimal policy. Figure~\ref{fig:normalization} shows that normalization is achieved in every framework, albeit at different rates. In particular, scenario MPC tends to produce smaller inputs in the initial phase, whereas the full-output feedback, which is initially constrained by the output value, tends to decay more rapidly. Figure~\ref{fig:comparison_srbs_vs_smpc} compares the percentage of runs that reached the optimal policy under scenario MPC and full-output feedback, demonstrating a significant performance improvement when employing the former, with performance further improving as the number of scenarios increases. Figure~\ref{fig:average_policy} highlights how the average policy, interpretable as a resulting stochastic policy, gets closer to the optimal deterministic one under scenario MPC. We emphasize, however, that this example is not intended to establish the superiority of scenario MPC over the full output feedback. Indeed, the latter retains a clear advantage in terms of computational efficiency and practical deployability. Rather, the purpose is to show that there exist normalization laws capable of achieving improved performance under model uncertainty, although possibly at a higher computational cost. This observation suggests that there remains room for improvement, and that the proposed framework may serve as a systematic tool for the design and analysis of more effective normalization laws beyond those currently available.
\begin{center}
    \begin{figure}[!b]
        \centering
        \includegraphics[width = 0.7\textwidth]{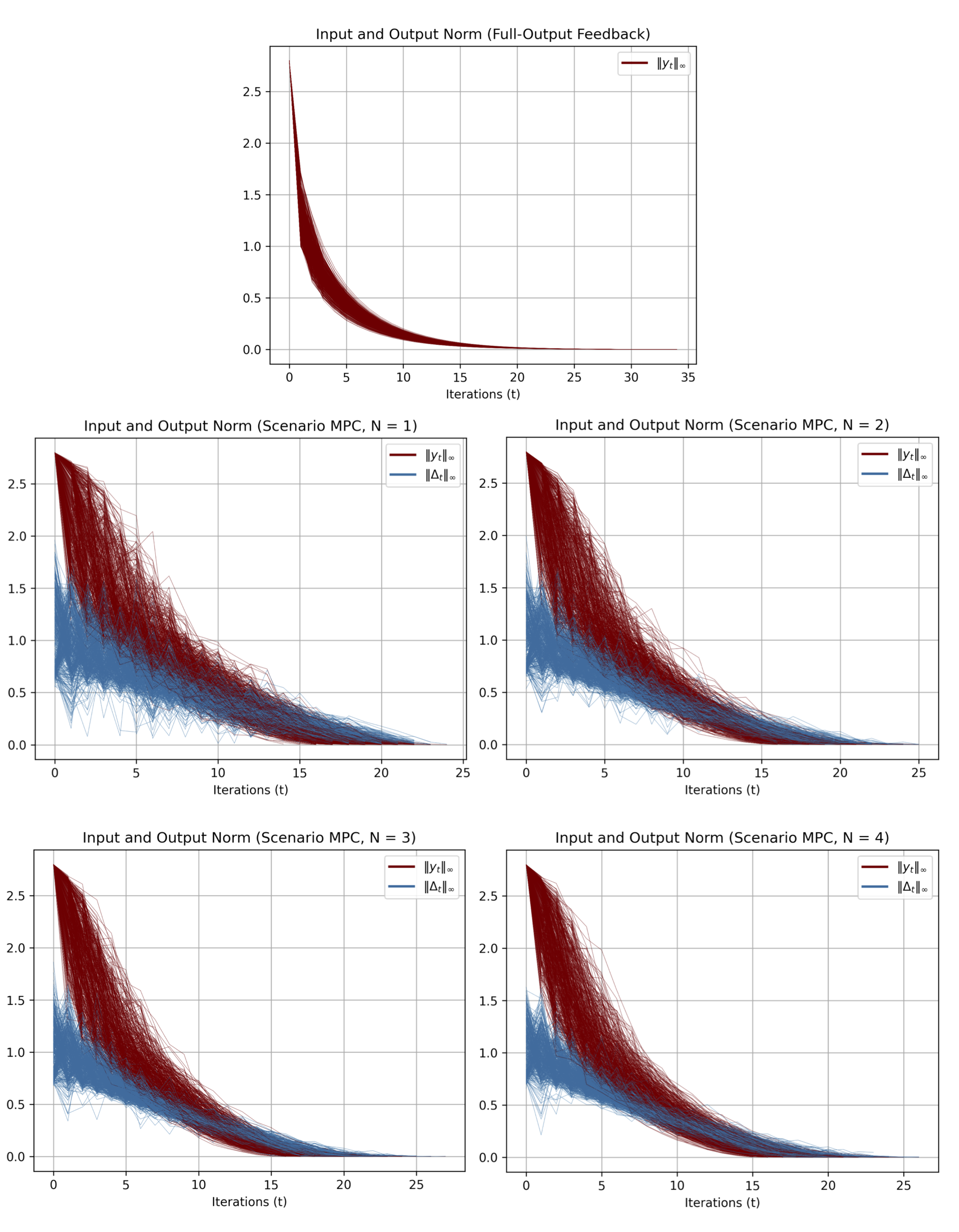}
        \caption{Input and Output supremum norms over $500$ Monte-Carlo simulations. Recall that, in the full-output feedback case, input and output coincide.}
        \label{fig:normalization}
    \end{figure}
\end{center}%
\begin{center}
    \begin{figure}[!t]
        \centering
        \includegraphics[width = 0.75\textwidth]{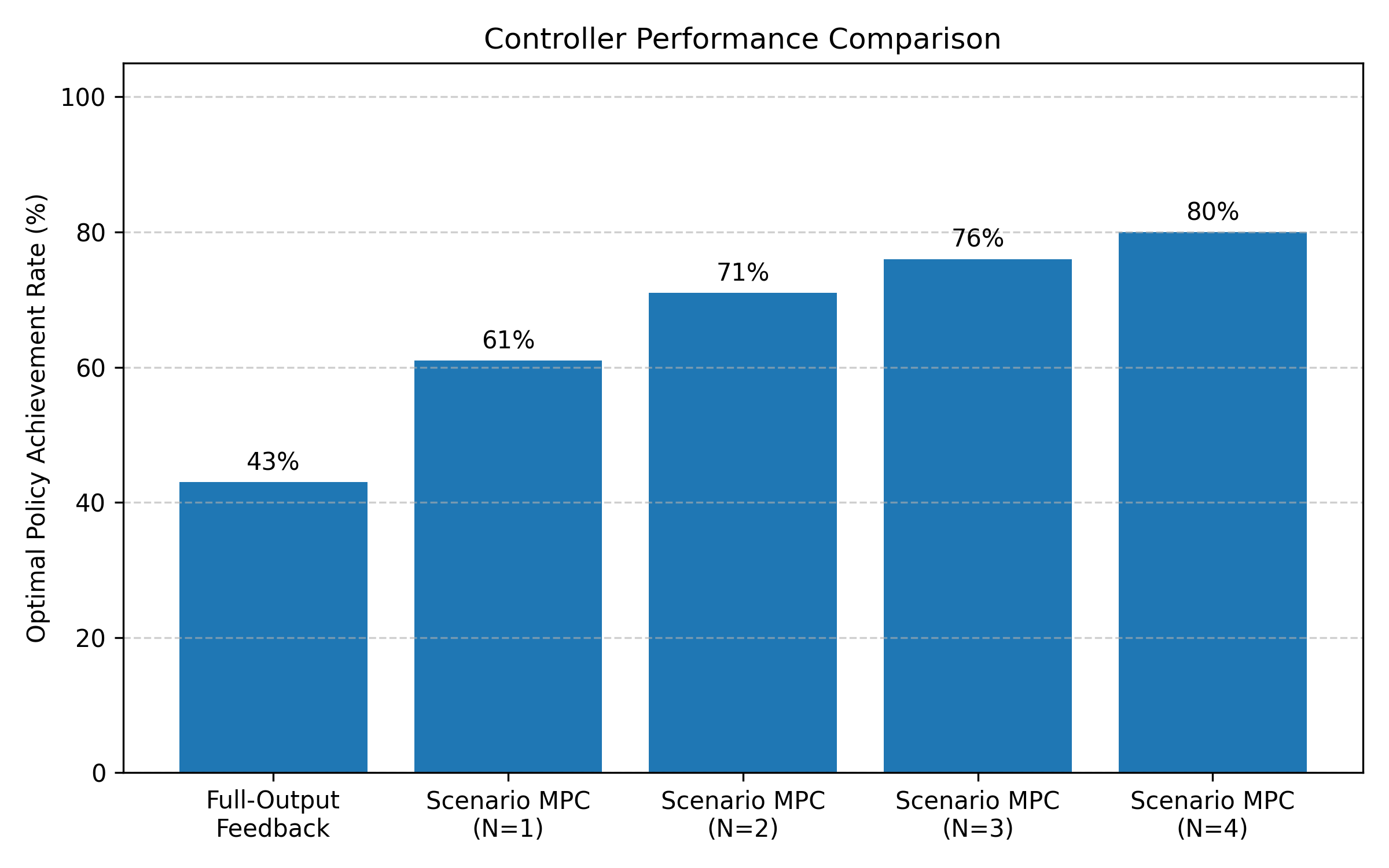}
        \caption{Comparison between the full-output feedback control law and scenario MPC in terms of the percentage of times the greedy policy computed from the normalized reward function is optimal.}
        \label{fig:comparison_srbs_vs_smpc}
    \end{figure}
\end{center}%
\begin{center}
    \begin{figure}[!t]
        \centering
        \includegraphics[width = 0.75\textwidth]{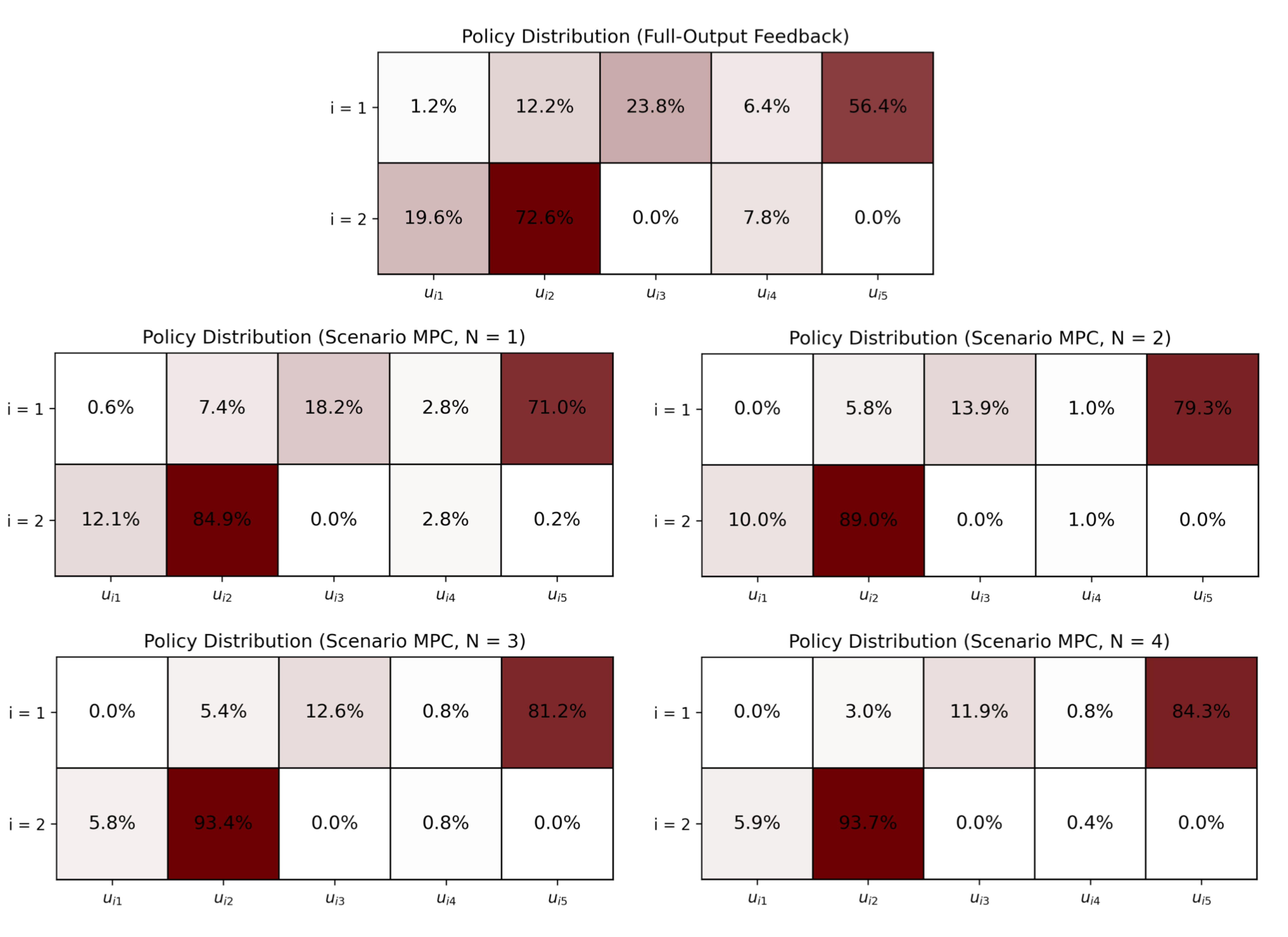}
        \caption{Average greedy policy over 500 Monte-Carlo simulations, expressed in percentage.}
        \label{fig:average_policy}
    \end{figure}
\end{center}%

\end{document}